\documentclass[11pt,a4paper]{article}

\usepackage[a4paper,top=3.8cm,bottom=3.8cm,left=3cm,right=3cm]{geometry}

\usepackage[T1]{fontenc}
\usepackage{amsmath,amssymb,mathtools,amsthm,amsopn}

\usepackage{hyperref}

\usepackage{titlesec}
\titleformat{\section}{%
\normalfont\large\bfseries}{\thesection.}{1em}{}
\titleformat{\subsection}{%
\normalfont\normalsize\bfseries}{\thesubsection.}{1em}{}

\usepackage{tocloft}
\usepackage{fancyhdr}
\usepackage{fourier-orns}

\usepackage{cite}

\renewcommand{\phi}{\varphi}
\newcommand{\perm}{\mathfrak{S}}
\newcommand{\U}{\mathcal{U}}
\DeclareMathOperator{\Div}{div}
\DeclareMathOperator{\cof}{cof}

\newcommand{\w}{\wedge}
\newcommand{\I}{\mathcal{I}}
\newcommand{\step}[1]{\par\medskip\noindent\it#1\rm}

\newcommand{\dcc}{d_{\textup{cc}}}
\DeclareMathOperator{\espo}{e}
\newcommand{\eap}{\espo_{\textup{ap}}}
\newcommand{\expap}{\exp_{\textup{ap}}}

\newcommand{\abs}[1]{\lvert#1\rvert}

\newcommand{\g}{\gamma}

\newcommand{\wh}{\widehat}
\renewcommand{\H}{\mathcal{H}}

\newcommand{\s}{\sigma}

\renewcommand{\P}{\mathcal{P}}

\renewcommand{\O}{\mathcal{O}}

\renewcommand{\cal}[1]{\mathcal{#1}}
\newcommand{\wt}{\widetilde}

\newcommand{\ol}{\overline}

\newcommand{\A}{\mathcal{A}}
\renewcommand{\d}{\delta}
\newcommand{\Eucl}{\textup{Euc}}
 
\newcommand{\p}{\partial}

\newcommand{\R}{\mathbb{R}}
\newcommand{\N}{\mathbb{N}}

\newtheoremstyle{pluto}  {}{}
{\slshape}  {}{\bfseries}  {.} {1ex}    {}

\newtheorem{theorem}{Theorem}[section]
\newtheorem{proposition}[theorem]{Proposition}

\newtheorem{lemma}[theorem]{Lemma}

\theoremstyle{pluto}
\newtheorem{definition}[theorem]{Definition}
\newtheorem{remark}[theorem]{Remark}
\newtheorem{example}[theorem]{Example}

\renewcommand{\d}{\delta}
\renewcommand{\t}{\tau}

\renewcommand{\a}{\alpha}
\renewcommand{\b}{\beta}
\newcommand{\loc}{\textup{loc}}
\DeclareMathOperator{\Lip}{Lip}
\DeclareMathOperator{\Span}{span}
\DeclareMathOperator{\ad}{ad}
\newcommand{\W}{\mathcal W}

\usepackage[dayofweek]{datetime}

\newcommand{\norm}[1]{\left\Vert#1\right\Vert}
\numberwithin{equation}{section}

\allowdisplaybreaks[3]
\let\oldbibliography\thebibliography
\renewcommand{\thebibliography}[1]{%
  \oldbibliography{#1}%
  \setlength{\itemsep}{0pt}%
}
\usepackage{enumitem}
\newenvironment{enumerate*}{\begin{enumerate}[noitemsep] }{\end{enumerate}}
\newenvironment{itemize*}{\begin{itemize}[noitemsep] }{\end{itemize}}
\newenvironment{description**}{\begin{description}[noitemsep] }{\end{description}}
\begin{document}

\title{Almost exponential maps and integrability results
\\
for a class of 
horizontally regular vector fields
}
\author{Annamaria Montanari \and Daniele Morbidelli}

\date{}

\maketitle

\begin{abstract}
   We consider a family $\H:= 
 \{X_1, \dots, X_m\}$ of $C^1$ vector fields in
$\R^n$ and we discuss  the associated $\H$-orbits.
Namely, we assume that our vector fields belong to a \emph{horizontal regularity} class and
we require that   a suitable \emph{$s$-involutivity} assumption holds. Then
we show that any $\H$-orbit $\O $ is a  $C^1$ immersed
submanifolds and it is an integral submanifold of the distribution generated by the family of all
commutators up to length $s$. Our main tool
is a class of \emph{almost exponential maps} of  which we discuss carefully  some precise 
first order  expansions.
\end{abstract}

  \tableofcontents

\section{Introduction and main results}

In this paper we  discuss the  integrability of distributions defined by  families of vector fields under  a higher
order \emph{horizontal regularity} hypothesis and assuming an
involutivity condition of order $s\in \N$.
The central tool we exploit is given by a class of \emph{almost exponential
maps} which we will analyze in details assuming only low regularity on the coefficients of the 
vector fields.

To start the discussion, fix 
a family $\H= \{X_1, \dots, X_m\} $  of at least Lipschitz-continuous vector fields. 
For   any $x\in \R^n$ define the Sussmann's
 \emph{orbit}, or   \emph{leaf}
\begin{equation}\label{orbitale}
\cal{O}_{\cal{H}}^x  : = \{ e^{t_1X_{j_1}} \cdots e^{t_p X_{j_p}}x:
p\in\N,  J:=(j_1,
\dots, j_p)\in \{1,\dots,m\}^p, t\in \Omega_{J,x}\},
\end{equation}
where for fixed $x\in\R^n$ we denote by 
 $\Omega_{J, x}\subset\R^p$ the open neighborhood of the origin where the map 
 $t\mapsto e^{t_1X_{j_1}} 
\cdots e^{t_p X_{j_p}}x$ is well defined.
We equip  the leaf  $\cal{O}_\H^x$   with the topology 
$\t_d$ defined by the Franchi--Lanconelli distance $d$; see \eqref{dicocco}.

Our purpose is to describe a regularity class of order  $s\ge 2$  
and a $s$\emph{-involutivity}
assumption that ensure that each orbit $\O_\H$ is a integral manifold of the distribution generated by the family~$\P:=\P_s:= \{Y_1, \dots, Y_q\}$ of all nested commutators of length at most $  s $ constructed from the original family $\H$. 
To give coordinates on $\O$ we shall use the  following 
\emph{almost exponential maps}.  Fix $s\ge 2$ and 
denote by $\P$ the aforementioned family of commutators.
Assign  to each $Y_j$ the
\emph{length} $\ell_j\le s$, just    its order. Then, let 
\begin{equation}\label{almostino} 
 E_{I,x}(h):= \expap(h_1 Y_{i_1})\cdots \expap(h_p Y_{i_p})x,                                                    
\end{equation}
where  $I= (i_1, \dots, i_p)$ is a multiindex which fixes $p$ commutators $Y_{i_1},\dots, Y_{i_p}\in\P$, $h\in\R^p$ belongs to a neighborhood of the origin and $p\in\{1, \dots, n\}$ is suitable.
See \eqref{appsto} for the definition of  the \emph{approximate exponential}  $\expap$.
We shall use the maps in~\eqref{almostino}  to construct charts, developing a higher order, nonsmooth, quantitative extension of some ideas appearing in a paper by Lobry; see~\cite{Lobry70}; see Theorem~\ref{viola} and Remarks~\ref{noterella} and~\ref{lobrao} below.

Here is a description of our regularity class.
Let $\H= \{X_1, \dots, X_m\}$ and let $s\ge 2$.  Assume that   $X_j=:f_j\cdot\nabla\in C^1_\Eucl$  for all $j$ (here and hereafter $C^1_\Eucl$ refers to Euclidean regularity). Assume also that   
for each $p\le s$  and $j_1, \dots, j_p\in\{1,\dots,m\}$, 
all derivatives 
$X_{j_1}^\sharp \cdots X_{j_{p-1}}^\sharp f_{j_p}$ exist and are locally  Lipschitz-continuous 
functions with respect to   distance $d$ associated to the vector fields. 
Here, following \cite{MontanariMorbidelli11a}, we denote by $X^\sharp f$ the Lie derivative along the vector field $X$ of the scalar function $f$. 
Moreover we require 
that  for any   commutator $Y_j=: g_j\cdot\nabla\in\P$, 
all maps of the form $g_j\circ  E_{I,x}$ are continuous for all $p\in\{1,\dots,n\}$, $I= (i_1, \dots, i_p)$
and $x\in\R^n$. \footnote{This condition is widely ensured for instance as soon as we assume that $g_j$ is continuous in the Euclidean topology, or at least in   the Sussmann's  orbit topology defined on $\O$ by the family $\H$; see \cite{Sussmann}.}

Furthermore, we require the following  $s$-involutivity condition. For any $X_j\in\H$ and for any $Y_k\in\P$ with maximal length $\ell_k=s$,
at any $x\in\Omega$ where the derivative $X_j^\sharp g_k(x)$ exists  one can write for suitable $b^i= b^i(x)$
\begin{equation}\label{integreria}\begin{aligned}
(\ad_{X_j} Y_{k})_x:= & ( X_j^\sharp g_k(x) - Y_k f_j(x))\cdot\nabla
= \sum_{ i=1}^{q}
b^i Y_{i,x}
\quad\text{with $b^i$ locally bounded.}\end{aligned}
\end{equation}
The class of vector fields satisfying all those assumptions will be denoted by $\A_s$; see Definition \ref{laint}, where
a more precise formulation of this assumption is described.
Note that in the smooth case we have $\ad_{X_j} Y_{k} = [X_j, Y_k]$ and ultimately \eqref{integreria}
is equivalent to the Hermann condition \cite{Hermann}
\begin{equation}\label{ajello}
 [Y_i, Y_j] = \sum_{1\le k\le q} c_{ij}^k Y_k,\quad\text{with $c_{ij}^k\in L^\infty_\loc$,  }
\end{equation}
which ensures that any  Sussmann's orbit $\O_\cal{P}$ of the
family  of commutators $\cal{P}$
is a integral manifold of the distribution generated by~$\P$.
If  furthermore $s=1$, then $\P= \H$ and \eqref{ajello} and \eqref{integreria} are the same.
Note that the appearance of operators of the form  $\ad_{X_j}Y_k$ 
is very natural in the framework of our almost 
exponential maps; see  the non-commutative calculus 
formulas discussed in 
\cite[Section 3]{MontanariMorbidelli11a}.

Here is the statement of our result.

\begin{theorem}\label{grossouno}
 Let $\H= \{X_1, \dots, X_m\}$ be a family of vector fields of class $\A_s$.
Then, for any $x_0\in\R^n$, the orbit $\O:= \O^{x_0}_\H$ with the topology $\t_d$ is a $C^1$ immersed  submanifold of $\R^n$
with tangent space $T_y\O = P_y$ for all $y\in \O$.
\end{theorem}

 Note that this result does not follow from
standard ones, because the commutators $Y_j$ are not assumed to be $C^1$ in the
Euclidean sense. In Example \ref{esempietto} we exhibit a family of vector fields where our theorem apply, but classical results do not. See also Remark~\ref{finalremark} for some further  comments.
 Furthermore, let us mention that if 
$s=1$, i.e.~$\H=\P$, then Theorem \ref{grossouno} is a consequence of 
the  Frobenius Theorem for singular $C^1$ distributions 
(it is well known to experts that in such case   one can prove that
 orbits are even $C^2$ smooth). Note that if $s=1$,  in
 \cite{MontanariMorbidelli11b} we proved a singular Frobenius-type theorem   
assuming only Lipschitz-continuity of the involved vector fields, generalizing part of  Rampazzo's results \cite{Rampazzo} to singular distributions; in fact, in \cite{MontanariMorbidelli11b}, orbits are~$C^{1,1}$.

On a technical level, the main tool we discuss is the approximate exponential~$E_{I,x}$
in~\eqref{almostino}.  Introduce the notation
$
 p_x:= \dim P_x :=    \dim \Span\{Y_{1}(x),\dots, Y_{q}(x)\}$ 
for  all   $x\in\R^n$. Fix $x$, take $p:= p_x$  commutators $Y_{i_1}, \dots, Y_{i_p}$, which are  linearly
independent
at $x$
and construct the map $E$, defined in \eqref{almostino}.
Then, under the hypotheses of Theorem \ref{grossouno},   we shall show that if the family $\H$ satisfies condition $\A_s$, then   $E$ is a $C^1_\Eucl$,  full   rank map in a
neighborhood of the origin
$0\in
\R^p$, whose derivative enjoys the following remarkable expansion
\begin{equation}\label{polacco}
\begin{aligned}E_*(\p_{h_k}) =   Y_{i_k}(E(h))  +\sum_{\ell_j=\ell_{i_k}+1}^{s}
a^j_k(h)
 Y_{j  }(E(h))
+ \sum_{i=1}^q  \omega_k^i(x,h)  Y_{i}(E(h)).
\end{aligned}\end{equation}
The functions $a_k^j$ and $\omega_k^i$ have a very precise  rate of
convergence to
$0$, as $h\to 0$ which will be specified in \eqref{sogliola} and
\eqref{merluzzo}. Note that an expansion of $E_*(\p_{h_k})$ can be obtained either with 
the Campbell--Hausdorff formula in the smooth case  (see \cite{Morbidelli98} or \cite{VaropoulosSaloffCosteCoulhon}), or in nonsmooth situations with the techniques of \cite{MM}. However, the expansions in  the mentioned  papers contain 
some remainders   
appearing either   as   formal series, or in integral form. Here we are able to express such reminders via the pointwise terms 
$\omega_k^j$, improving  all previous  results.  Note also that we are improving    the
 mentioned papers both from a regularity standpoint and 
 because here we do not assume the H\"ormander 
condition. 
At the authors' knowledge, expansion \eqref{polacco} with  precise estimates on $a_k^j$ and $\omega_k^i$  is new even in the smooth case. As a final remark, observe that   Theorem \ref{deduciamo} contains an explicit  detailed    proof of the fact that the map $E$ is $C^1$ smooth, avoiding any use of the Campbell--Hausdorff formula. Note that, even if the vector fields are smooth, such maps are not much more than $C^1$; see Remark \ref{tredodici}-(ii).

The useful information one  can  extract from
\eqref{polacco} is that $
E_*(\p_{h_k}) \in P_{E(h)}$ (note that we are interested to situations where the  inclusion    $P_{E(h)}\subset\R^n$ is strict);
see Theorem~\ref{deduciamo} for a precise statement. Observe that, if
$O\subset\R^p$ is a small open set containing the origin, then $E(O)$
is a $C^1$ submanifold of $\R^n$ and    \eqref{polacco} shows that $
T_{E(h)}E(O) \subseteq P_{E(h)}$ for all $h$.  This is the starting point to prove that $\O_\H^x$ is a  integral manifold of the distribution generated by~$\P$.
Another fact we need to prove is that  
the dimension of  $P_y:= \Span\{Y_j(y): 1\le j\le q\}$ is  constant if $y$ belongs to a fixed orbit $ \O_\H^x$. This is obtained by means of a nonsmooth quantitative curvilinear  version of the original Hermann's argument inspired to the work of Nagel, Stein and Wainger~\cite{NagelSteinWainger} and Street~\cite{Street}.

To conclude this introduction, we 
give some references and motivations to study our almost exponential maps~$E$.
Such  maps   appear in \cite{NagelSteinWainger}, and were used by the authors to show   equivalence between different control distances; see also \cite{VaropoulosSaloffCosteCoulhon}.
 More recently they have revealed to be a useful tool to study Poincar\'e inequalities (see \cite{LM}),
subelliptic Sobolev spaces (see \cite{Danielli91,Morbidelli98,Coulhon01,MM}),   and 
    geometric  theory of   Carnot--Carath\'eodory spaces (see 
\cite{MontiMorbidelli02,FerrariFranchi03,Vittone10}). 
Finally, note that the precise  expansion \eqref{polacco} will be a fundamental  tool in the companion  paper \cite{MontanariMorbidelli11c}, where we shall prove a Poincar\'e inequality on orbits 
for a family of vector fields satisfying an integrability condition.

\section{Preliminaries}
\label{preliminarmente}

\paragraph{Vector fields    and
the control distance.}
Consider a
family    of vector fields $\H=\{ X_1, \dots, X_m\}$
 and assume that   $X_j\in C^1_\Eucl(\R^n)$ for all $j$. Here and later
$C^1_\Eucl$ means $C^1$ in the Euclidean sense.
Write $X_j = :f_j\cdot\nabla$, where $f_j\colon\R^n\to
\R^n$.
The vector field $X_j$, evaluated at a point $x\in \R^n$, will be denoted
by  $X_{j,x}$ or $ X_j(x)$.
All the vector fields in this paper are always defined on  the whole
space $\R^n$.

Define the Franchi--Lanconelli distance \cite{FL}
\begin{equation}\label{dicocco}
\begin{aligned}
 d(x,y)&:=
 \inf \Big\{r>0: \text{$y=e^{t_1 Z_1}\cdots e^{t_\mu Z_\mu} x$ for some $\mu\in
\N$}
 \\& \qquad\qquad \qquad
\text{where $\sum \abs{t_j}\le 1$
 with $Z_j\in  r\cal{H} $} \Big\}.
\end{aligned}
\end{equation}
Here  and hereafter we let $r\H:=  \{r X_1,\dots, r X_m\}$ and
$\pm r\H:=  \{\pm r X_1,\dots, \pm r X_m\}$. The topology associated with $d$ will be denoted with $\t_d$.
We denote instead by $\dcc$ the standard  \emph{Carnot--Carath\'eodory} or 
\emph{control} distance  (see Feffermann--Phong \cite{FeffermanPhong81} and Nagel--Stein--Wainger~\cite{NagelSteinWainger}).  In the present paper we shall make a prevalent use
of the distance~$d$.
It is well known  that  
$\t_d$ is (possibly strictly) stronger than
the topology $\t_\Eucl|_\O$ received by $\cal{O}$ from $\R^n$. See \cite[Chapter
3]{Berhanu} and
\cite[Example 5.5]{Agrachev}.

In view  of the mentioned examples, we need to use the  broad definition of submanifold; see \cite{Chevalley,KobayashiNomizu}.
Below, if $\Sigma\subset\R^n$, we denote by $\t_\Eucl|_\Sigma$ the induced topology.

\begin{definition}
 [Immersed submanifold] Let $\Sigma\subset\R^n$ and let $\t\supseteq\t_\Eucl|_\Sigma$ be a 
topology on
 $\Sigma$.
We say that $\Sigma$ is a $C^k$ submanifold if $\Sigma $ is connected and for all $x\in \Sigma$ there is  
 $\Omega\in\tau$, open neighborhood of $x$ such that $\Omega$ is a $C^k$ graph.
If moreover  $\t=\t_\Eucl|_\Sigma$ then we say that $\Sigma$ is an \emph{embedded} submanifold.
\end{definition}

\paragraph{Horizontal regularity classes.}
Here we define   our notion
of horizontal regularity in terms of the distance $d$. Note that we  \emph{do
not} use the control distance $\dcc$.
\begin{definition}\label{montgomery}
  Let    $\H :=\{  X_1, ,\dots,X_m \}$ be a family of
  vector fields, $X_j\in C^1_\Eucl$. Let $d$ be their   distance \eqref{dicocco}
Let $g:\R^n\to \R$. We say that $g$ is \emph{$d$-continuous}, and we write
$g\in
C^0_\H(\R^n)$, if for all $x\in \R^n$, we have $
g(y)\to  g(x)$, as
$d(y,x)\to 0$.
We say that $g:\R^n\to \R$    is \emph{$\H$-Lipschitz} or \emph{$d$-Lipschitz}
in $A\subset\R^n$ if
\begin{equation*}
 \Lip_\H(g; A):= \sup_{ {x,y\in A,\; x\neq y}} \frac{|g(x)-
g(y)|}{d(x,y)}<\infty.
\end{equation*}
We say  that  $g\in C_\H^1(\R^n)$ if the derivative
$X_j^\sharp g(x): = \lim_{t\to 0}(f(e^{tX_j}x) - f(x))/t$ is a $d$-continuous
function
for any $j=1, \dots, m$.
 We
say that $g\in C^k_{\cal{H}}(\R^n)$ if all the derivatives
$X_{j_1}^\sharp\cdots X_{j_p}^\sharp g$ are $d$-continuous for $p\le k$ and $j_1,
\dots,
j_p\in\{1,\dots,m\}$. If all the derivatives
$X_{j_1}^\sharp\cdots X_{j_k}^\sharp g$ are $d$-Lipschitz on each $\Omega$
bounded set in the
Euclidean metric, then we say that $g\in
C^{k,1}_{\cal{H},\loc}(\R^n)$.
Finally, denote the usual Euclidean Lipschitz constant of $g$  on $A\subset\R^n$
by $\Lip_{\Eucl}(g; A)$.
\end{definition}

 We will  usually  deal with vector fields which are of class at least
$C^1_\Eucl\cap C^{s-1,1}_{\H,\loc}$, where $s\ge 1$ is a suitable integer.
In this case it   turns out that  commutators up to the order $s$ can be
defined; see Definition \ref{ffff}.
In the companion paper \cite{MontanariMorbidelli11a} we study  several issues 
related with this definition.

\paragraph{Definitions   of commutator.}
Our  purpose  now  is to show that, given a family $\H$ of vector fields with
$X_j\in C^{s-1,1}_{\H,\loc} \cap C^1_\Eucl$, then commutators  can be   defined
up to
length $s$.

For any $\ell\in \N$,
denote by $\W_\ell:  = \{ w_1\cdots w_\ell: w_j\in \{1,\dots, m\}\}$ the words
of
length $\abs{w}:=\ell$ in the alphabet $1,2,\dots,m$. Let also $\perm_\ell $ be
the
group of  permutations of $\ell$ letters.
Then for all $\ell\ge 1$, there are functions $\pi_\ell:\perm_\ell\to \{-1,0,1\}$ such that
\begin{equation}
\label{commototo}
[A_{w_1},[A_{w_2},\dots
[A_{w_{\ell-1}},A_{w_\ell}]]\dots ]=  \sum_{\s\in
\perm_\ell}\pi_\ell(\s) A_{\s_1(w)} A_{\s_2(w)}\cdots A_{\s_\ell(w)},
\end{equation}
for all $A_1, \dots,
A_m:V\to V$  linear operators on a  vector space $V$.
See \cite{MontanariMorbidelli11a}   for a more formal definition and an in-depth discussion.

We are now ready to define commutators for vector fields in our regularity
classes.
\begin{definition}
[Definitions of commutator]\label{ffff} 
 Given a family  $\H=\{X_1,\dots,  X_m\}$ of vector fields  of class
$C^{s-1,1}_{\H,\loc}\cap C^1_\Eucl $,   define for any function $\psi\in C^1_\H$ the operator \(X_j^\sharp\psi(x) :=
\cal{L}_{X_j}\psi(x),\)  the Lie derivative. Let also  $
X_j\psi(x) := f_j(x) \cdot \nabla \psi(x)$ where $\psi\in C^1_\Eucl$.
Moreover, let
\begin{equation*}
\begin{aligned}
f_w &: =  \sum_{\s\in \perm_\ell } {\pi}_\ell(\s)\big( X_{\s_1(w)}\cdots
X_{\s_{\ell -1}(w)}
f_{\s_\ell(w)} \big)\quad\text{for all $w$ with $\abs{w}\le s$,}
 \\
X_w \psi  & :=[X_{w_1}, , \dots, [X_{w_{\ell-1}}, X_{w_\ell}]] \psi
: = f_w \cdot\nabla \psi\quad\text{for all $\psi\in C^1_\Eucl\quad \abs{w}\le
s$,}
\\
 X_w^\sharp      \psi&
: = \sum_{\s\in \perm_\ell } {\pi}_\ell(\s)
X_{\s_1(w)}^\sharp\cdots X_{\s_{\ell-1}(w)}^\sharp
X_{\s_\ell(w)}^\sharp\psi \quad\text{for all
$\psi\in C^{\ell}_{\cal{H}}$\quad $\abs{w}\le s-1$.
}
\end{aligned}
\end{equation*}
Finally,   for any $j\in\{1,\dots, m\}$ and
$w$ with $1\le \abs{w}\le s$, let
\begin{equation}\label{addio}
 \ad_{X_j} X_w \psi : = (X_j^\sharp  f_w - f_w\cdot\nabla f_j )\cdot\nabla \psi
=(X_j^\sharp f_w - X_w f_j) \cdot\nabla\psi \quad\text{for all  $\psi\in
C^1_\Eucl$.}
\end{equation}
\end{definition}
Non-nested commutators are precisely defined in \cite{MontanariMorbidelli11a}.

\begin{remark} \begin{itemize*}\item Let $Z\in \pm \H $.  If $\abs{w}\le s-1$, then there are no
problems in defining $\ad_Z X_w$. More precisely, in    \cite{MontanariMorbidelli11a} we
show that  $\ad_Z X_w =[ Z, X_w]$.   If instead
$\abs{w}=s$, then
 the function
$t\mapsto
f_w(e^{t Z}x)$ is Euclidean Lipschitz. In particular it is differentiable for
a.e.~$t$. In other words, for any fixed $x\in \R^n$, the limit
$
 \frac{d}{dt} f_w(e^{t Z}x) = : Z^\sharp f_w(e^{tZ}x )$
exists for a.e.~$t$ close to $0$.  Therefore the pointwise derivative $Z^\sharp
f_w(y)$ exists for almost
all
$y\in \R^n$ and ultimately $\ad_Z X_w$ is defined almost everywhere.

\item  Both our definitions of commutator, $X_w$ and
$X_w^\sharp$
are well
posed from an
algebraic point of view, i.e.~they satisfy  antisymmetry and the Jacobi identity; see \cite{MontanariMorbidelli11a}.

\item In \cite{MontanariMorbidelli11a} we will also  recognize that the first order operator $X_w$ agrees with
$X_w^\sharp$ against functions   $\psi\in C^{s-1,1}_{\cal{H},\loc} \cap
C^1_\Eucl$
as soon as  $|w|\le s-1$. 

\end{itemize*}
\end{remark}

\paragraph{The integrability class $\A_s$.}
\begin{definition}[Vector fields of class $\mathcal{A}_s$]\label{laint}
 Let $\H= \{X_1, \dots, X_m\} $ be a family in the regularity class
 $C^1_\Eucl\cap
C^{s-1,1}_{\H,\loc} $.  We say that the family $\cal{H}$ belongs to the class $\A_s$
if, fixed an open  bounded set $\Omega\subset \R^n$,
there is $C_0>1$ such that the following holds: for
any $Z\in \pm\H$,   for any
word $w$ with $\abs{w}=s$, for each $x\in
\Omega$ and for a.e.~$t\in [-C_0^{-1}, C_0^{-1}]$,
there are
coefficients $b^v\in \R$ such that
\begin{align}\label{arte}
 \ad_{Z} X_w (e^{tZ}x) & = \sum_{1\le \abs u\le s}b^u X_{u}(e^{tZ}x)\quad\text{with}
\\ \label{artefatto}  \abs{b^u } & \le C_0 \qquad\text{for all $u$ with $1\le
\abs{u}\le s$; }
\end{align}
finally assume that if $1\le \abs{w}\le s$,  for all $p\in\{1,\dots,n\}$, for any $I\in\I(p,q)$, $x\in \R^n$,
we have at any $h^*$ where $E_{I,x}$ is defined
\begin{equation}\label{sussmann}
  f_w(E_{I,x}(h))  \longrightarrow
f_w(E_{I,x}(h^*))  \quad\text{as
$h\to h^*$.}
\end{equation}
\end{definition}

\begin{remark}
\begin{itemize*}

\item Assumption \eqref{sussmann} will be used only once, in
 \eqref{sosso}, but
it
is essential in order to ensure that the  almost exponential maps we define
later are actually $C^1_\Eucl$ smooth.
It is easy to check that  assumption  \eqref{sussmann} is  satisfied as soon
as $f_w: (\O_\H,\t_\H)\to \R$
is continuous, where $\t_\H$ denotes the Sussmann's  orbit topology defined by the
family $\H$, see \cite{Sussmann}. 
  Note that at this stage assumption \eqref{sussmann} is not ensured by the $d$-Lipschitz continuity of $f_w$.

\item
Conditions \eqref{arte} and \eqref{artefatto} scale nicely. Namely,
letting 
for all $r\le
1$, $\wt Z= r Z$, $\wt X_w  = r^{\abs{w}}X_w$ with $\abs{w}=s$, we have
\begin{equation}\begin{aligned}\label{artusi} \ad_{\wt Z} \wt X_w (x) = \sum_{1\le \abs u\le
s} \wt b^u
\wt X_{u}(x) \quad\text{ where
  $\abs{\wt b^u} \le C_0 r\le C_0$ for all $u$.}
\end{aligned}\end{equation}
\item 
Let      $\H $ be a family of vector fields in the class $C^1_\Eucl\cap C^{s-1,1}_{\H,\loc}$ satisfying the  H\"ormander 
bracket-generating condition  of step $s$ and  assume that each  $f_w $  with $\abs{w}\le s$ is continuous in the Euclidean sense.
Then 
$\H$ satisfies $\A_s$. The constant $C_0$
in \eqref{artefatto} depends also on a positive lower bound on  $\inf_\Omega\abs{\Lambda_n(x,1)}$, see \eqref{panacea}. 
 This case is discussed in   \cite[Section~4]{MontanariMorbidelli11a}. 

\item
The pathological  vector fields $X_1= \p_{x_1} $ and $X_2=
e^{-1/{x_1}^2}\p_{x_2},
$ in spite of their $C^\infty$ smoothness, do not satisfy  \eqref{artefatto} for any  $s\in \N$.
\end{itemize*}
\end{remark}

Let $\Omega_0\subset\R^n$ be a fixed open set, bounded in the Euclidean metric.
Given a family   $\H$  of vector fields of
 class $C^1_\Eucl\cap C^{s-1,1}_{\H,\loc}$,  introduce the
constant
\begin{equation}
\label{lipo}
\begin{aligned}
L_0  :&  = \sum_{j_1,\dots, j_s=1}^m\Big\{ \sup_{\Omega_0 }
\Big(  |f_{j_1}| +
|\nabla
f_{j_1}|
+
  \sum_{p\le s} |X_{j_1}^\sharp \cdots
 X_{j_{p-1}}^\sharp f_{j_p}|\Big)
\\  & \qquad
\qquad \qquad +\Lip_\H(X_{j_1}^\sharp\cdots
X_{j_{s-1}}^\sharp f_{j_s};   \Omega_0)\Big\}.
  \end{aligned}
       \end{equation}
We shall always choose points  $x\in\Omega\Subset\Omega_0$ and we fix a
constant
$t_0>0$ small enough
to ensure that
\begin{equation}\label{hello}
e^{\t_1 Z_1}\cdots e^{\t_N Z_N}x\in \Omega_0
\quad \text{if $x\in \Omega$,
$Z_j\in \H$, $\abs{\t_j}\le t_0$ and $N\le N_0  $,}
\end{equation}
 where
$N_0$ is a suitable  constant which depends  on the    data  $n, m$ and
$s$.

\begin{proposition}[measurability] \label{misod} Let $\H$ be a family  of
class $\A_s$. Let $\abs{w} =s$ and let
$Z\in \pm\H$,
 Then for any $x\in \Omega$ we can write
 \begin{equation}
 \ad_{Z} X_w(e^{t Z}x) = \sum_{1\le\abs{v}\le s} b^v(t)
X_v(e^{t  Z}x) \qquad \text{for a.e. $t\in (-t_0, t_0)$,}
\end{equation}
where the functions $t\mapsto b^v(t)$ are \emph{measurable} and for a.e.~$t$ we have $\abs{b^v(t)}\le
C_0$, where $C_0$ denotes  the constant in \eqref{artefatto}.
\end{proposition}
\begin{proof}  The statement can be proved arguing as in   
\cite[Proposition~4.1]{MontanariMorbidelli11a}.
\end{proof}


\paragraph{Wedge products and $\eta$-maximality conditions.}
Following \cite{Street}, denote by
$\cal{P}   := \{ Y_1, \dots, Y_q\} = \{ X_w: 1\le \abs{w} \le s\}$
the family of commutators of length at most $s$.  Let
$  \ell_j\le s$  be the length of $Y_j$ and
write $Y_j=:g_j\cdot\nabla$.
Define for any $p,\mu\in \N$, with $1\le p\le \mu$,
$\I(p,\mu)  := \{I=(i_1, \dots, i_p): 1\le i_1<i_2< \cdots<i_p\le \mu\}  $.
  For each $x\in\R^n$ define
 $ p_x:= \dim  \Span   \{ Y_{j,x} : 1\le j\le q\}.$
Obviousely, $p_x\le \min\{n,q\}$. Then for any $p\in \{1,\dots,\min\{n,q\}\}$,
let
\begin{equation*}
 Y_{I,x} : = Y_{i_1,x}\wedge\cdots\wedge Y_{i_p,x}\in
{\textstyle\bigwedge}_p
T_x\R^n\sim {\textstyle\bigwedge}_p\R^n \quad\text{for all $I\in \I(p,q)$,}
\end{equation*}
and, for all $K\in \I(p,n)$ and $ I\in
\I(p,q)$
\begin{equation}\label{girocollo} 
\begin{aligned}
 Y_I^K(x)& : = dx^K(Y_{i_1}, \dots, Y_{i_p}) (x)
: = \det(g_{i_\a}^{k_\b})_{\a,\b=1,\dots, p}.
  \end{aligned}
\end{equation}
Here we let $dx^K:=dx^{k_1}\wedge \cdots \wedge d x^{k_p}$  for any
$K=(k_1,\dots, k_p)\in
\I(p,n)$.

The family $e_K:= e_{k_1}\wedge\cdots\wedge e_{k_p}$, where  $K\in\I(p,n)$,
gives an othonormal basis of $\bigwedge_p\R^n$, i.e. $\langle e_K, e_H\rangle
= \delta_{K,H}$ for all $K,H$. Then we  have the orthogonal decomposition
$
 Y_I(x)  =\sum_{K}Y_J^K(x) e_K\in {\bigwedge}_p \R^n
$, so that the number
\[ |Y_I(x)| : =\bigl(\sum_{K\in \I(p, n)}Y_I^K(x)^2\bigr)^{1/2}  =
\abs{Y_{i_1}(x)\wedge\cdots\wedge Y_{i_p}(x)}
\]
gives the  $p$-dimensional volume  of the parallelepiped  generated by
$Y_{i_1}(x), \dots, Y_{i_p}(x)$.

Let  $I= (i_1, \dots, i_p)\in \I(p, q)$ such that $\abs{Y_I}\neq 0$. Consider
the linear system $\sum_{k=1}^p\xi^k
Y_{i_k}= W$, for some
$W\in\Span\{Y_{i_1}, \dots, Y_{i_p}\}$. The Cramer's
rule  gives the unique solution
\begin{equation}
\label{cromo} \xi^k = \frac{\langle Y_I, \iota^k(W)
Y_I\rangle}{\abs{Y_I}^2}\quad\text{for each $k=1,\dots, p$,}
\end{equation}
where we let $\iota^k_W Y_I:=
\iota^k(W)
Y_I:= Y_{(i_1,\dots, i_{k-1})}\wedge
 W\wedge Y_{(i_{k+1},\dots,i_p)}.$

Let    $r>0$.  Given $J\in \I(p,q)$, let $\ell(J):=
\ell_{j_1}+ \cdots + \ell_{j_p}$. Introduce the vector-valued function
\begin{equation}\label{panacea} 
\begin{aligned}
\Lambda_p(x,r)& := \bigl(  Y_J^K (x)r^{\ell(J)} \bigr)_{{J\in \I(p, q)}, K\in
\I(p, n)}
=: \bigl( \wt Y_J^K (x) \bigr)_{{J\in \I(p, q),\,  K\in\I(p, n)}},
\end{aligned}
\end{equation}
where we adopt the tilde notation $\wt Y_k  = r^{\ell_k}Y_k $ and its obvious
generalization for wedge products. Note that $\abs{\Lambda_p(x,r)}^2 =
\sum_{I\in \I(p, q)}
r^{2\ell(I)}\abs{Y_I(x)}^2$.

\begin{definition}
 [$\eta$-maximality] Let $x\in\R^n$,  let $I\in \I(p_x,q)$ and $\eta\in (0,1)$.
We say
that $(I,x,r)$ is
$\eta$-maximal if
$\abs{Y_I (x)} r^{\ell (I)} >\eta \displaystyle \max_{J\in \I(p_x, q)} \abs{Y_J(x)}r^{\ell(J)}.
$
\end{definition}
Note that, if $(I,x,r)$ is a candidate to be  $\eta$-maximal with
 $I\in \I(p,q) $, then by definition it \emph{must} be    
$p=p_x=\dim\Span\{Y_j(x):1\le j\le q\}$.


\paragraph{Approximate exponentials of commutators.} \label{capuozzo}
Let $w_1, \dots, w_\ell\in \{1,\dots,m\}$.
Given $\t>0$, we define, as in \cite{NagelSteinWainger,Morbidelli98} and
\cite{MM},
\begin{equation}\label{navetta}
 \begin{aligned}
 C_\t( X_{w_1})& := \exp(\t X_{w_1}),
 \\ C_\t( X_{w_1}, X_{w_2})& :=\exp(-\t X_{w_2})\exp(-\t X_{w_1})\exp(\t
X_{w_2})\exp(\t X_{w_1}),
 \\&\vdots
  \\C_\t( X_{w_1}, \dots, X_{w_\ell})&
:=C_\t( X_{w_2}, \dots, X_{w_\ell})^{-1}\exp(-\t X_{w_1}) C_\t( X_{w_2}, \dots,
X_{w_\ell})\exp(\t X_{w_1}). \end{aligned}
 \end{equation}
Then let
\begin{equation}\label{appsto} 
\eap^{tX_{w_1 w_2\dots w_\ell}} :=  \expap(t X_{{w_1 w_2\dots w_\ell}}):=
\left\{\begin{aligned}
& C_{t^{1/\ell}}(X_{w_1}, \dots, X_{w_\ell} ), \quad &\text{ if $t\geq 0$,}
\\
&C_{|t|^{1/\ell}}(X_{w_1}, \dots, X_{w_\ell} )^{-1}, \quad &\text{ if $t<0$.}
                 \end{aligned}\right.
\end{equation}
By standard ODE theory,    there is $t_0$  depending on $\ell,  \Omega$,
$\Omega_0$,  $ \sup\abs{ f_j } $ and  $\sup\abs{\nabla f_j}  $     such that
$\exp_*(t X_{{w_1 w_2\dots w_\ell}})x\in\Omega_0$   for any $x\in
 \Omega$ and $|t|\le t_0$.
   Define, given $I= (i_1,\dots,i_p)\in\{1,\dots,q\}^p $,   $x\in \Omega$ and
$h\in\R^p$, with
$|h|\le C^{-1}$
     \begin{equation}
\begin{aligned} \label{hhh}E_{I,x}(h)& :=\expap(h_1 Y_{i_1})\cdots
\expap(h_p
Y_{i_p})(x)
\\
\bigl\|h\bigr\|_I & : =\max_{j=1,\dots,p}|h_j|^{1/\ell_{i_j} }\quad
\text{and}\quad   Q_I(r):
=\{h\in\R^p:\norm{h}_I < r\}.
\end{aligned}\end{equation}

\paragraph{Gronwall's inequality.} We shall refer several times to the
following standard fact: for all $a\ge 0$, $b>0$,
$T>0$ and $f$ continuous on $[0,T]$,
\begin{equation}\label{grammatica}
 0\le f(t)\le a t+ b\int_0^t f(\t) d\t \quad\forall \;
t\in [0, T]\quad\Rightarrow
\quad f(t)\le \frac{a}{b} (e^{b t}-1)\quad\forall\, t\in[0,T].
\end{equation}


\section{Approximate exponentials and regularity of \texorpdfstring{$A_s$}{As} 
orbits} \label{kobayashi}
Let $\H = \{X_1, \dots, X_m\}$ be a family of $\A_s$ vector fields in $\R^n$.
The main  purpose of this section is to prove that any $\H$-orbit $\cal{O}_\H$
with the topology $\t_d$ generated by the distance~$d$ is a $C^1$ integral manifold of  the distribution generated by $\P$.
Recall our  usual notation
$ \cal{P} : = \{ Y_j: 1\le j\le q\}$,
$ P_x
:=\Span\{Y_{j,x}: 1\le j\le q\}$ and $p_x:= \dim P_x.$

\subsection{Geometric properties of orbits}\label{fiveone}
In this subsection   we  look at the properties of orbits $\O_{\H}$ for vector
fields of class $\A_s$.
First we study how the geometric determinants $\wt Y_J^K$ change along a given
orbit $\O_\H$. The argument  we use   is known, see for
instance
\cite{TaoWright03,MM} and
especially  \cite{Street}. However,  we need to address
some issues which appear due to our  low regularity assumptions.
Ultimately, we will
show that the positive integer
   $p_x$ is
constant as   $x\in\O_\H$.

Below we shall  use the  following notation:  given $r>0$, we
let $\wt Y_j =
r^{\ell_j} Y_j = :
\wt g_j\cdot\nabla$ and $\wt Z = r Z$, if $Z\in \pm \cal{H}$. Let also $\wt Y_J^K:=r^{\ell(J)}Y_J^K$,
 where the notation
for
$Y_J^K$ has been introduced in   \eqref{girocollo}.
\begin{lemma}  \label{lili} Let $\H$ be a family of vector fields of class
$\A_s$. Let  $ p\in\{1,\dots, q \wedge n\}$. Let $x\in
\Omega$ and $r_0>0$ so that $B_d(x,r_0)\subset\Omega_0$.
 Let  $J\in \I(p , q)$, $K\in \I(p, n)$,  $r\in (0, r_0]$ and   $\wt
Z\in \pm
r\H$. Then the
function
$[-1,1]\ni t\mapsto \wt Y_J^K(e^{t \wt Z}x)$ is
  Lipschitz continuous
 and 
there is  $C>1$ depending on $C_0$ and $L_0$ in \eqref{artefatto}
and \eqref{lipo}
such that
\begin{equation*}
 \Bigl|\frac{d}{dt}\wt Y_J^K (e^{t \wt Z}x) \Bigr|\le
C\abs{\Lambda_p(e^{t \wt Z}x,
r)} \quad\text{ for
a.e. $t\in (-1, 1)$}.
\end{equation*}
\end{lemma}

\begin{proof}
Denote $\gamma_t:=e^{t \wt Z}x$ and let $t, \t\in (-1, 1)$.  Then
\begin{equation*}
 \begin{aligned}
|\wt Y_J^K(\g_\t)  -\wt Y_J^K(\g_t)|
& = \Big| \sum_{1\le \a\le p} dx^{K}(\dots,   \wt Y_{j_{\a+1}}(\g_t), \wt
Y_{j_\a}(\g_\t)-
\wt Y_{j_\a}(\g_t),
  \wt Y_{j_{\a+1}}(\g_t), \dots )\Big|
\\&\le C \abs{\t-
t},
\end{aligned}
\end{equation*}
where $C$ depends on $L_0$ in \eqref{lipo}.  Then $t\mapsto \wt Y_J^K(\g_t)$ belongs to 
$\Lip_\Eucl(-1, 1)$.
The estimate for the Lipschitz constant here   is quite rough and it can be
refined through a computation of the derivative. Indeed, we claim that for a.e.
$t\in (-1, 1)$ we have
\begin{equation}\begin{aligned}\label{maino}
 \frac{d}{dt} \wt Y_J^K(\g_t)
&  =
\sum_{\substack{1\le\a\le p\\ \ell_{j_\a}\le
s-1}}  dx^K(\dots, \wt Y_{j_{\a-1}}, [{\wt Z}, \wt Y_{j_\a}], \wt Y_{j_{\a+1}},
\dots,
\wt Y_{j_p})(\g_t)
\\& +\sum_{\substack{ 1\le \a\le  p\\
\ell_{j_\a}= s}}   \sum_{1\le \b\le q} b_{\a
}^\b(\g_t)
dx^K(\dots,
\wt Y_{j_{\a-1}}, \wt Y_\b, \wt Y_{j_{\a+1}}, \dots, \wt Y_{j_p})(\g_t)
\\&+ \sum_{1\le \g\le n}\sum_{1\le \b\le p} \p_\g  \wt f^{k_\b}
dx^{(k_1,\dots,k_{\b-1},\g,k_{\b+1},\dots, k_p)} ( \wt Y_{j_1},\dots,\wt
Y_{j_p})(\g_t)
\\&=: (A)+(B)+(C),
\end{aligned}
\end{equation}
where we wrote $\wt Z= \wt f\cdot\nabla\in C^1_\Eucl$ and   $b_\a^\b$ are
measurable
functions with  $\abs{b_{\a }^\b}\le C_0$.
To prove \eqref{maino}, observe that, if
  $\ell(Y_{j_\a})\le s-1$,  then
 $t\mapsto \wt Y_{j_\a}(\g_t)$ is
$C^1_\Eucl(-1, 1)$ and
\begin{equation*}
\begin{aligned}
 \lim_{\t\to t}\frac{ \wt Y_{j_\a}(\g_\t)- \wt Y_{j_\a}(\g_t)}{\t- t} & = {\wt
Z}^\sharp
\wt g_{j_\a}(\g_t) \cdot\nabla =
[{\wt Z} , \wt Y_{j_\a}](\g_t )  +\wt Y_{j_\a} \wt f(\g_t)\cdot\nabla
\quad
\text{for all $t\in [-1, 1]$. }
\end{aligned}
\end{equation*}
Note that here we used \cite[Theorem~3.1]{MontanariMorbidelli11a}   to claim that $\ad_{\wt Z}\wt Y_{j_\a}
= [\wt Z, \wt Y_{j_\a}]$.
If instead $\ell(Y_{j_\a})= s$, then
for almost any $t$ we have
\begin{equation}\label{gigos}
\begin{aligned}
 \lim_{\t\to t}
\frac{\wt Y_{j_\a}(\g_\t)- \wt Y_{j_\a}(\g_t)}{\t- t}
& = {\wt Z}^\sharp  \wt g_{j_\a}(\g_t) \cdot\nabla = \ad_{{\wt Z}}\wt
Y_{j_\a}(\g_t )+
\wt Y_{j_\a}
\wt f (\g_t)\cdot\nabla
\\& =\sum_{\b=1}^q b_\a^\b(t)  \wt Y_\b(\g_t) +
\wt Y_{j_\a}\wt f(\g_t)\cdot\nabla.
\end{aligned}
\end{equation}
In the first equality we   used the  definition of $\ad$. Here  $\wt
Y_{j_\a}\wt f:=
\wt g_{j_\a}\cdot\nabla \wt f$, is well
defined.
In the second line we used   Proposition \ref{misod}.
The term $\wt Y_{j_\a}\wt f$,  in
view of Lemma \ref{ollo} gives the third line of \eqref{maino}.

Next we estimate each line of \eqref{maino},
starting with $(A)$.
\begin{equation*}
\begin{aligned}
  |(A)|& \le \big|dx^K(\dots, \wt Y_{j_{\a-1}}(\g_t), [{\wt Z},
\wt Y_{j_\a}](\g_t), \wt Y_{j_{\a+1}}(\g_t), \dots ) \big|
\le C
\abs{\Lambda_p(\g_t, r)},
\end{aligned}
\end{equation*}
for all $t\in[-1, 1]$. Estimate is correct even if $\Lambda_p(\g_t, r)=0$.
To estimate  $(B)$, recall that $|b_{\a }^\b|\le C$. Then,
for all
$t\in[-1,1]$,
\begin{equation*}
\begin{aligned}
 |(B)| &\le \sum_{1\le\a\le p} \sum_{1\le \b\le q} \big| dx^K(\dots,\wt
Y_{j_{\a-1}}, \wt Y_\b,
\wt Y_{j_{\a+1}}, \dots)\big|
\le C\abs{\Lambda_p(\g_t,r)}.
\end{aligned}
\end{equation*}
Finally  the estimate of $(C)$ is easy and takes the form
 \begin{align*}
 \abs{(C)}& \le \sup_{B_d(x, r)}|\nabla \wt f|\max_{K\in
\I(p,n)}|\wt Y_J^K(\g_t)|  \le C\abs{\Lambda_p(\gamma_t,
r)}\quad \text{if $\abs{t}\le 1$.}\qedhere
\end{align*}
\end{proof}

The previous lemma immediately implies the following proposition.
\begin{proposition} \label{valida}
Let $\H$ be a family in the regularity class $\A_s$.
 Let $x\in\Omega$, let $r\le r_0$, where $r_0$ is small enough so that $B_d(x,
r_0)\subset\Omega_0$. Let $\gamma(t):=\g_t$ be a piecewise integral curve of $\pm r\H$ with $\gamma(0)=x$.
 Let $p\in\{1,\dots, q\wedge n\}$. Then we have
\begin{equation}\label{spray}
 \big|\Lambda_p(\g(t),r)- \Lambda_p(x,r)\big|\le
\abs{\Lambda_p(x,r)}\,(e^{Ct}-1)\quad\text{for all $t\in[0,1]$}.
\end{equation}
In particular, 
if $p= p_x$ and   $(I, x, r)$ is $\eta$-maximal,
then  
\begin{equation}\label{damettere}
\abs {\wt Y_J(\g(t))- \wt Y_J(x)}
\le \frac{C t}{\eta}\abs{\wt Y_I(x)}
\quad\text{for all $J\in \I(p,q)\quad t\in[0,1]$.}
\end{equation}
Finally, if $x,y$ belong to the same orbit, then  
$p_x   =  p_{y}$.
\end{proposition}
\begin{remark}\label{osservo}
As a consequence of the proposition and of the Cramer's rule \eqref{cromo}, if
$(I,x,r)$ is $\eta$-maximal, then $(I,
y, r)$ is
$ C^{-1}\eta$-maximal for all $y\in B_d(x, C^{-1}\eta r)$ and we may write for
all such $y$ and for any  $j\in\{1,\dots,q\}$
\begin{equation}
\label{giorgetto}
\wt Y_{j,y} = \sum_{k=1}^p \frac{b_j^k }{\eta}\wt Y_{i_k,
y},
 \end{equation}
where $\abs{b_j^k}\le C$.
\end{remark}

\begin{remark}\label{optima} 
 Proposition \ref{valida} shows that the oscillation of determinants
$\Lambda_p$ on a ball is controlled in terms of the value of $\Lambda_p$ at the
center of the ball. It is not true that the oscillation of a single vector
field on a ball can be controlled by its value at the center of the ball.
 For instance, we can take   the vector fields
 $X=\p_x$ and $Y=y\p_y + x\p_x$. Look at the   ball $B((0,y), r)$, where
 $0< y \ll r$.
Note that
$(r,y) $ belongs to such ball, but the oscillation $\abs{Y(0,y)- Y(r,y)} \sim
r$  can not be controlled with the value $\abs{Y(0,y)} = \abs{y}$.
\end{remark}

\begin{proof}[Proof of Proposition \ref{valida}]
 (See \cite{TaoWright03,MM,Street}). Let $p\in \{1,\dots,q\wedge n\}$.
By Lemma~\ref{lili}, the map  $t\mapsto \Lambda_p(\g_t,r) $ is  Lipschitz.
 Moreover,
we have for a.e. $t\in[0, 1]$,
\begin{equation*}
 \begin{aligned}
\Big|\frac{d}{d t}\Lambda_p(\g_t,r)\Big|
&=
\Big|\Big(\frac{d}{d t}\wt Y_J^K(\g_t)\Big)_{\substack{J\in \I(p,q)\\K\in\I(p,n)}}\Big|
\le C \abs{\Lambda_p(\g_t,r)},
\end{aligned}
\end{equation*}
by Lemma \ref{lili}. Then the Gronwall's inequality  \eqref{grammatica} provides immediately  
the required estimate \eqref{spray}.
Note that this implies that
if $\Lambda_p(x,r)=0$, then  $\Lambda_p(\g_t,r)=0$ for all $t\in[0, 1]$.
Estimate \eqref{damettere} follows immediately.

Let now $x $ and $y$ be a couple of points on 
 the same leaf $\cal{O}_\H$.  
Let $1\le p\le q\wedge n$ and let $I\subset\R$ be an interval.
Let $I= [a,b]$ and take  $\gamma:I\to \R$  a
piecewise integral curve of the vector fields $X_j$ with $\gamma(a)= x$ and $\gamma(b)=y$. 
Let $A_p:=\{t\in
I:\abs{\Lambda_p(\gamma(t))}=0\}$. 
Note that   $A_p$ is closed, because it is the zero
set of the continuous function $I \ni t\mapsto
\abs{\Lambda_p(\gamma(t))}\in\R$. The set $A_p$  is also open by estimate \eqref{spray}.
Therefore, either
$A_p=\varnothing$ or $A_p= I$ and the proof is concluded.
\end{proof}

The fact we are going to establish in the following theorem will have a key
role in Subsection~\ref{pastorius}, when we shall study our almost exponential
maps $E$. See
 Remark~\ref{noterella} below.
\begin{theorem}
\label{viola}
Let $\H$ be a family of vector fields of class $\A_s$. Let $(I, x, r)$ be
$\eta$-maximal  where $x\in \Omega$, $r\le r_0$,  $I\in \I(p_x, q)$ and
$\eta\in (0,1)$. Denote
$\wt U_j:= r^{\ell_{i_j}} Y_{i_j}$ for $j=1,\dots, p:= p_x$ and $\wt Z:= r Z\in \pm
r\H $. Then there is $C>0$ depending on $L_0$ and $C_0$ in
\eqref{lipo} and \eqref{artefatto} so that 
\begin{equation}\label{dentro}
e^{-t\wt Z}_*(\wt U_{j, e^{t\wt Z}x})\in P_x\quad\text{for all $t$ with
$\abs{t}\le C^{-1}\eta$}.
\end{equation}   Moreover, if we
write, for a given test function $\psi\in C^1_\Eucl(\R^n)$,
 \begin{equation}\label{75}
\wt U_j(\psi e^{-t   \wt Z})(e^{t \wt  Z}x)   = :
\sum_{k=1}^p
\big(\delta_j^k + \theta_j^k(t) \big) \wt U_k\psi(x),
\end{equation}
then    we have
\begin{equation}\label{zast} 
 |\theta_j^k(t)| \le \frac{C |t|}{\eta} \quad\text{for all
$j,k=1,\dots,p$\qquad
 $  \abs{t} \le C^{-1}\eta$}.
\end{equation}
Finally,  for any commutator $\wt Y_h := \wt g_h\cdot\nabla$, where  $h\in
\{1,\dots,q\}$, we
have at  any  $t\in (-C^{-1}\eta ,  C^{-1}\eta)$
\begin{equation}
\label{909}
\wt Y_h(\psi e^{-t\wt Z})(e^{t\wt Z}x) = \sum_{k=1}^p \frac{b_h^k(t)}{\eta}\wt
U_k\psi(x), \end{equation}
where  $  |b_h^k(t)|\le C$ if $\abs{t}\le C^{-1}\eta$.
\end{theorem}
\begin{remark}\label{noterella}
The geometric interpretation of \eqref{dentro} tells  that $e^{-t \wt
Z}_*P_{e^{t\wt Z}x} = P_x$, i.e. the tangent map of the $C^1$ diffeomorphism $e^{-t\wt Z}$
maps the (candidate) tangent bundle  $\cup_x P_x $ to the orbit $\cal{O}$
 to itself (we say ``candidate'' because we do not know yet that  $\O$
is a
manifold).
Theorem \ref{viola} has an important consequence.
Namely, in   in Theorem \ref{dascalare}, it will enable us to
show that integral remainders have in fact a  pointwise form.
Ultimately, we will apply such property in Theorem  \ref{deduciamo} to show that
$E_*(\p_{h_k})\in P_{E(h)}$.
\end{remark}

\begin{remark}\label{lobrao}
The proof below is inspired to an argument due to Lobry; see \cite[Lemma~1.2.1]{Lobry70}.
 Here we generalize 
such argument to a higher order, nonsmooth situation and we get more quantitative estimates.
See also  
 \cite{Lobry76} and the related  discussion 
 by  Balan
\cite{Balan}; see finally the paper~\cite{Pelletier}, for an up-to-date
bibliography on the subject.
Note that Lobry's idea is  also used in
\cite[Lemma 5.15]{Agrachev}.
\end{remark}

\begin{proof}[Proof of Theorem \ref{viola}]
Without loss of generality,   we can  work with positive values of $t$.
First, we  differentiate the left-hand side of \eqref{75}. If $\ell_{i_j}\le
s-1$,
then we use \cite[Theorem~2.6-(a) and Theorem~3.1-(ii)]{MontanariMorbidelli11a}  which give
\begin{equation}\label{popo}
\begin{aligned}
 \frac{d}{dt}\wt U_j(\psi e^{-t \wt Z})(e^{t \wt Z} x)& =[\wt Z, \wt
U_j](\psi e^{-t
\wt Z})(e^{t\wt Z}x)
= \sum_{k=1}^p \frac{b_j^k(t)}{\eta} \wt  U_k(\psi e^{-t\wt Z}) (e^{t\wt
Z}x),
\end{aligned}
\end{equation}
provided that $0<t\le C^{-1}\eta $. Here $\abs{b_j^k(t)}\le C$. In last
equality we used
\eqref{giorgetto} with $\wt Y_h = [\wt Z,
\wt U_j]$.

   If instead
$\ell_{i_j}= s$, then we need first \cite[Theorem 2.6-(b)]{MontanariMorbidelli11a}, 
then \eqref{sussmann} and Proposition \ref{misod} in the present paper. This gives
 for a.e. $t\in[0, C^{-1}\eta]$
\begin{equation}\label{popo2}
\begin{aligned}
\frac{d}{dt}\wt U_j (\psi e^{-t \wt Z})(e^{t\wt Z} x)& =  \sum_{1\le h\le
q } b_j^h(t) \wt Y_h
(\psi e^{-t
 \wt Z})(e^{t\wt Z}x) \quad \text{by \eqref{giorgetto}}
\\&=
\sum_{1\le h\le q}\sum_{1\le k\le p} b_j^h(t)  b_h^k(t)\frac{ 1}{\eta} \wt
U_k
(\psi e^{-t \wt Z})(e^{t \wt Z}x)
\\&=:\sum_{1\le k\le p} \frac{b_j^k(t)}{\eta} \wt U_k(\psi e^{-t \wt Z})(e^{t
\wt Z}x)
\end{aligned}
\end{equation}
provided that  $0<t\le C^{-1}\eta $. In this formula
$b_j^h$, $b_h^k$ and $b_j^k$ denote measurable   functions, bounded in term of
the admissible constants $C_0$ and $ L_0$.

By elementary ODE theory, for any fixed $\psi$, the    functions $t\mapsto \wt
U_j(\psi e^{-t\wt
Z})(e^{t\wt
Z}x) $ with $j=1,\dots,p$ are uniquely determined by their value
$\wt U_j\psi(x)$ at $t=0$.
Moreover, if we denote by $(a_j^k(t))\in \R^{p\times p}$
the solution of the Cauchy problem
\begin{equation}\label{odofusi}
\dot a(t)= \frac{b(t)}{\eta} a(t) \quad\text{with}\quad  a(0)= I_p\in\R^{p\times p},
                         \end{equation} then we can write
\begin{equation}
e_*^{-t \wt Z}(\wt U_{j, e^{t\wt Z}x})\equiv \wt U_j(\psi e^{-t \wt Z})(e^{t\wt
Z}x)
= \sum_{k=1}^p a_j^k(t) \wt U_k\psi(x).
\end{equation}
Then we have proved \eqref{dentro}. 
The Cramer's rule \eqref{cromo} confirms that the coefficients $a_j^k(t)$ are
unique for each $t$.

To estimate the functions
 $\theta_j^k:= a_j^k(t)- \delta_j^k   $,  where $a_j^k$ satisfy \eqref{odofusi},
it
suffices to use estimate  $|b_{j}^k(t)|\le C$ if $0\le t \le C^{-1}\eta $.
The
Gronwall inequality \eqref{grammatica} gives
\(
 |a_j^k(t) - \delta_j^k|\le  C |t|/\eta \)   for all
$j,k=1,\dots, p$ and $
0< t\le C^{-1}\eta .$ Therefore \eqref{zast} follows.

To obtain the  proof of \eqref{909} it suffices to repeat the computation in
\eqref{popo}  starting from $\wt Y_h$ instead of $\wt
U_j$. This ends the proof.
\end{proof}

Under the hypotheses of Theorem \ref{viola}, iterating the argument,  we get
for all $x\in \Omega$, $\mu\le N_0$ (see \eqref{hello}),  $j\in\{1,\dots,p\}$ and $Z_1,\dots,
Z_\mu\in\H$,
\begin{equation}\label{jappo}
 \wt U_j(\psi e^{-t_1 \wt Z_1}\cdots e^{-t_\mu \wt Z_\mu})(e^{t_\mu \wt
Z_\mu}\cdots e^{t_1
\wt Z_1}x) = \sum_{1\le k\le p} (\d_j^k +\theta_j^k(t))\wt U_k\psi(x)
\end{equation}
where
$\abs{\theta(t)}\le  C \abs{t}/\eta$,
as soon as $\sum_{j=1}^\mu\abs{t_j}\le C^{-1}\eta$.
Moreover, for each $h\in \{1,\dots,q\}$, we get, if $x\in \Omega $, for the
same values of $(t_1,\dots, t_\mu)$ and for almost  all $\t\in (- C^{-1}\eta,
C^{-1}\eta)$,
\begin{equation*}
\begin{aligned}
\frac{d}{d\t}\wt Y_h&
(\psi e^{-t_1 \wt Z_1}  \cdots e^{-t_\mu \wt Z_\mu}e^{-\t \wt X})
( e^{\t \wt X}e^{t_\mu \wt
Z_\mu}\cdots e^{t_1
\wt Z_1}x)
\\&= \ad_{\wt X}\wt{Y}_h(\psi e^{-t_1 \wt Z_1}\cdots e^{-t_\mu \wt Z_\mu} e^{-\t
\wt X})
(e^{\t \wt X}e^{t_\mu \wt
Z_\mu }\cdots e^{t_1
\wt Z_1}x)
=\sum_{k=1}^p \frac{b_k(x,t,\tau )}{\eta}\wt U_k\psi(x),
\end{aligned}
\end{equation*}
where $\abs{b_k(x,t,\t)}\le C$ for a.e.~$\t$. Here  $X\in  \H$. If we do not
care about maximality and choose $r= 1$, we get,
for any fixed $(t_1, \dots, t_\mu)$ with  $\sum_j\abs{t_j}\le C^{-1}$ and for
almost all $\t  $ with $\abs{\t}\le C^{-1}$,
\begin{equation}
 \label{jpn4}
\begin{aligned}
\frac{d}{d\t}  Y_h & (\psi e^{-t_1  Z_1}\cdots e^{-t_\mu   Z_\mu}e^{-\t X})
(e^{\t X}e^{t_\mu
Z_\mu}\cdots e^{t_1
  Z_1}x)
\\ & =
 \ad_{  X} {Y_h}  (\psi e^{-t_1  Z_1}\cdots e^{-t_\mu   Z_\mu} e^{-\t X})
(e^{\t X}e^{t_\mu
Z_\mu}\cdots e^{t_1
  Z_1}x)
\\& 
=\sum_{1\le j\le q}  b_j(x,t,\tau)  Y_j\psi(x),
\end{aligned}
\end{equation}
where $\abs{b_j (x,t,\t)}\le C$ for a.e.~$\t$. Here again $x\in \Omega$ and
$\psi\in C^1_\Eucl$
is a test function.
Formula \eqref{jpn4}
  will be referred to later.

\subsection{Derivatives of almost exponential maps and regularity of orbits}
\label{pastorius}
In this subsection we get several information on  the derivatives of 
the
approximate
exponentials $E_{I,x,r}$ associated with a family $\H$ of $\A_s$ vector
fields and we
 show that each orbit $\cal{O}$ with topology $\t_d$ is a $C^1$
immersed submanifold of $\R^n$ with $T_y\O= P_y$ for all $y\in\O$.
 We will  tacitly but heavily rely
 on the results of \cite[Section~3]{MontanariMorbidelli11a},
namely on  formulae
\begin{equation}\label{usbus} 
 \ad_{X_{v_1}}\cdots \ad_{X_{v_k}}X_w = X_{vw}\quad\text{for all $v,w$ such
that $\abs{v}+ \abs{w} = k+\abs{w}\le s$}
\end{equation}
These formulae have   a key
role. In the
proof of  Theorem \ref{dascalare} below, we shall  follow the arguments of 
 \cite[Theorems 3.4 and  3.5]{MM},
 modifying
everywhere the remainders $O_{s+1}$ in \cite{MM}
with our remainders defined in \cite{MontanariMorbidelli11a}.
This will give us a  formula with integral remainder, see  \eqref{grall}.
Then, using the results of Subsection~\ref{fiveone}, we shall show  that such
integral remainder can be
specified in a pointwise form.

\begin{theorem}\label{dascalare}
 Let $1\le \abs{w}=:\ell\le s$, take  $x\in\Omega$ and  $ t\in [0,
t_0]$, where  $t_0$ is small enough to ensure that $C_t x\in
\Omega_0$ for all $t\in[0, t_0]$.
Let $C_t=C_t(X_{w_1}, \dots, X_{w_\ell})$ be the map defined in
\eqref{navetta}. Fix a test function $\psi\in
C^1_\Eucl(\R^n)$. Then we have
\begin{equation*}
  \frac{d}{dt}\psi (C_t x )= \ell t^{\ell-1} X_w \psi(C_tx) +\sum_{|v|=\ell+1}^s
a_v t^{|v|-1}
  X_v\psi(C_tx)
+t^s\sum_{\abs{u}=1}^s b_u(x,t)X_u\psi(C_t x),
\end{equation*}
and
\begin{equation*} 
\begin{aligned}
  \frac{d}{dt}\psi(C_t^{-1} x) = & - \ell t^{\ell-1} X_w \psi(C_t^{-1}x)
+\sum_{|v|=\ell+1}^s \ol a_v
  t^{|v|-1}
  X_v \psi(C_t^{-1}x)
\\&
  + t^s \sum_{\abs{u}=1}^s \ol b_u(x,t)X_u\psi (C_t^{-1}x).
\end{aligned}
\end{equation*}
 Both the sums on $v$ are empty if $\abs{w}=s$. Otherwise,
 we have the
cancellations
$\sum_{\abs{v}=\ell+1}(a_v+\ol a_v)f_v(x)=0$ for all $x\in \Omega$.
The (real) coefficients $b_u$ and $\ol b_u$ are bounded in terms of
the constants $L_0$ and $C_0$ in \eqref{lipo} and \eqref{artefatto}.
\end{theorem}

\begin{remark}
 As already observed, the theorem just stated improves
\cite[Theorem~3.5]{MM}, both  because we relax regularity  assumptions and because   we devise a pointwise
form of the
remainders. In particular, choosing as $\psi$ the identity function, we see
that the  remainder  belongs to the subspace $P_{C_tx}=\Span\{Y_{j,C_t x}:
j=1,\dots, q\}$ which can be
a strict subspace of $\R^n$.
\end{remark}

\begin{proof}[Proof of Theorem~\ref{dascalare}] We prove the statement for $t>0$. By \cite[Theorem~3.5]{MM}, we know that
 \begin{equation}\label{grall}
  \frac{d}{dt}\psi (C_t x )= \ell t^{\ell-1} X_w \psi(C_tx) +\sum_{|v|=\ell+1}^s
a_v t^{|v|-1}
  X_v\psi(C_tx)
+O_{s+1}(t^s,\psi, C_tx),
\end{equation}
where the numbers  $a_v$ are suitable algebraic coefficients. 
Note that formula \eqref{grall} in \cite{MM} is proved for smooth vectro fields. Using \eqref{usbus} and changing
everywhere the remainders in \cite{MM} with the  remainders introduced in 
\cite[Subsection~2.1]{MontanariMorbidelli11a},
 one can check that all computations fit to our setting. 
Therefore, we only need to
deal with the integral remainders introduced and discussed in~\cite{MontanariMorbidelli11a}.   
Concerning such remainders, recall that
\begin{equation*}
\begin{aligned}
 O_{s+1}(t^s, \psi,C_t x)&  =\text{(sum of terms like)} \int_0^t \omega(t,
\t)\frac{d}{d\t} X_v(\psi\phi^{-1}e^{-\t Z})(e^{\t Z}\phi  C_t x) d\t
\end{aligned}
\end{equation*}
where $\abs{v}=s$, $\phi = e^{t Z_1 }\cdots e^{t Z_\nu}$ and $Z, Z_j\in
\pm \H $.
Next, by \eqref{jpn4}, we may write for a.e.~$\t$
\begin{align*}
 \frac{d}{d\t} X_v(\psi\phi^{-1}e^{-\t Z})(e^{\t Z}\phi  C_t x)
&
=
\sum_{1\le \abs{u}\le s} b_u(x,t,\tau)X_u\psi(C_t x),
\end{align*}
where
for any $t,x$  the functions
$\t\mapsto b_u(x,t,\tau)$ are
measurable and satisfy $\abs{ b_u(t,\tau,x)}\le C$ for a.e.~$\t$. Therefore we
get
\begin{equation*}
 \sum_{1\le \abs{u}\le s}\int_0^t \omega(t,
\t) b_u(x,t,\tau)  d\t
 \; X_u \psi(   C_t
x)=: t^s \sum_{1\le \abs{u}\le s}  b_u(x,t) X_u\psi(C_t x),
\end{equation*}
where $\abs{b_u(x,t)}\le C$ for all $x\in\Omega$ and $\abs{t}\le t_0$.
This ends the proof.
\end{proof}

Our purpose now is  to study the maps
\begin{equation}\label{fivetwenty}
 E(h):= E_{I,x,r}(h) := \expap(h_1 \wt Y_{i_1})\cdots
\expap(h_p \wt Y_{i_p})= \eap^{h_1 \wt U_1}\cdots
\eap^{h_p \wt U_p}x
\end{equation}
where $1\le p\le q$, $I\in \I(p,q)$, $\wt U_k:= \wt Y_{i_k}$ and $d_k:= \ell_{i_k}$. We
always take $x\in \Omega$ and $h$ sufficiently close to the origin so that
$E(h)\in \Omega_0$, see \eqref{hello}.

Some elementary properties of $E$ are contained in the following lemma. Without
loss of generality we choose $r=1$ and  $I= (1,\dots, p)$.

\begin{lemma}\label{torrette}
 The map $ h  \mapsto \eap^{h_1 Y_1}\cdots\eap^{h_p Y_p}x=:E_{I,x}(h)$ satisfies
for $x,x^*\in \Omega$ and $h, h^*\in B_\Eucl(C^{-1})$
\begin{equation}\label{galileo}
 \abs{E_{I,x}(h)- E_{I, x^*}(h^*)} \le
C\big(\bigl\|h-h^*\bigr\|_I+\abs{x-x^*}\big).
\end{equation}
Moreover, for any $w$ with $1\le \abs{w}\le s$,  the
function $F_{X_w}\colon[-C^{-1},C^{-1}]\times \Omega\to \R^{n\times
n}$, defined as
$F_{X_w}(t,x ):= \nabla _x \eap^{t X_w}(x),$ is continuous.
\end{lemma}

\begin{proof}
 Observe first that, since each   $Z\in \pm \H$ is $C^1_\Eucl$,    by
the Gronwall inequality we have
\begin{equation}\label{baio}
 \abs{e^{\t Z}y - e^{ \t_0 Z}y_0} \le  C\big(\abs{y- y_0} +
\abs{\t-\t_0}\big)\quad\text{for all $y,y_0\in \Omega\quad
\abs{\t},\abs{\t_0}\le C^{-1}$. }
\end{equation}

Next, assume first that  $t\ge t^*\ge 0$. Write $\eap^{t X_w}x= e^{\t Z_1}\cdots
e^{\t Z_\nu}x$,  where $Z_1, \dots, Z_\nu\in \pm\H$ are suitable, see
\eqref{appsto},  and $\t = t^{1/\ell}$, with $\ell:=\abs{w}$. Then iterating \eqref{baio} we get
\begin{equation*}
\begin{aligned}
 \bigl|\eap^{t X_w}x - \eap^{t^* X_w}x^*\bigr|
&= \bigl|e^{\t Z_1}\cdots e^{\t
Z_\nu}x - e^{\t^* Z_1}\cdots e^{\t^*
Z_\nu}x^*\bigr|
\le C\bigl(\abs{x- x^*} + \abs{t - t^*}^{1/\ell} \bigr).
\end{aligned}
\end{equation*}
If instead $t>0>t^*$, then we get
\begin{align*}
 \bigl|\eap^{t X_w}x - \eap^{t^* X_w}x^*\bigr|
&\le \abs{\eap^{t X_w}x - x} +
\abs{ x ^*- \eap^{t^* X_w}x^*}+ \abs{x- x^*}
\\&\le
C\bigl(\abs{t}^{1/\ell} + \abs{t^*}^{1/\ell} + \abs{x- x^*}\bigr)
\le C\bigl(
\abs{t-t^*}^{1/\ell} +   \abs{x- x^*}\bigr).
\end{align*}
This shows \eqref{galileo} for $p=1$. Iterating one gets the general case.

Next we prove   existence and continuity of the
derivative $F_{X_w}$. Assume first
 that $t\ge 0$ and decompose
$\eap^{t
X_w} x= e^{t^{1/\ell }Z_1}\cdots
e^{t^{1/\ell} Z_\nu}x$, where $\ell= \abs{w}$ and $Z_1, \dots,Z_\nu \in \pm\H$
are suitable. Euclidean regularity of the vector fields  
$Z_j$ implies that
the functions $(\t,y)\mapsto F_{Z_j}(\t, y):= \nabla_y e^{\t Z_j}y$ are continuous if $y\in \Omega$
and $\abs{\tau}$ is small.
Therefore,  the chain rule gives
\begin{align*}
F_{X_w}(t,x)&= \nabla_x \eap^{t X_w} (x)
\\& = F_{Z_1}(t^{1/\ell},
e^{t^{1/\ell}Z_2}\cdots e^{t^{1/\ell }Z_\nu}x) F_{Z_2}(t^{1/\ell},
 e^{t^{1/\ell}Z_3}\cdots  (x))
\cdots F_{Z_\nu}(t^{1/\ell} ,x).
\end{align*}
Thus $F_{X_w}\bigr|_{[0, C^{-1}]\times\Omega}$ is continuous.
Note that $F_{X_w}(0,x)= I_n$ for all $x$. An analogous
argument shows that $F_{X_w}\bigr|_{[-C^{-1},0]\times\Omega}$ is continuous and
concludes the proof.
\end{proof}

At this point we may deduce the following result. See \eqref{fivetwenty} for
notation on the map~$E$.
\begin{theorem}\label{deduciamo}
Let $\H$ be an $\A_s$ family. Let $x\in \Omega$ and let $r\in (0,r_0)$.
Fix $p\in\{1,\dots,q\}$ and $I\in \I(p, q)$. Then the function $E_{I,x,r}$ is $C^1$ smooth on
$B_\Eucl(C^{-1})$.
Moreover, for all $h\in B_\Eucl(C^{-1})$ and for any $k\in\{1,\dots,p\}$
we have $E_*(\p_{h_k})\in P_{E(h)}$ and we can write
\begin{equation}\label{nhb}
E_*(\p_{h_k}) = \wt U_{k,E(h)}  +\sum_{\ell_j=d_k+1}^{s} a^j_k(h)
\wt Y_{j, E(h)}
+ \sum_{i=1}^q  \omega_k^i(x,h)\wt Y_{i,E(h)},
\end{equation}
where, for some $C>1$ depending on $L_0$ and $C_0$ in \eqref{lipo} and
\eqref{artefatto}, we have
\begin{align}\label{sogliola}
       \abs{a_k^j (h)}& \le C \bigl\|h \bigr\|_I^{\ell_j-d_k}\quad\text{for all
$h\in B_\Eucl(C^{-1})$}
\\ \label{merluzzo}  \abs{\omega_i(x,h)} &\le C \bigl\|h \bigr\|_I^{s+1-  d_k}
\quad\text{for all $h\in B_\Eucl(C^{-1})\quad x\in\Omega$}.
\end{align}
\end{theorem}
\begin{proof}
For notational simplicity we delete everywhere the tilde. In fact,  the
statement holds uniformly in
$r\in (0, r_0)$, where $r_0$ depends on the already mentioned constants $L_0$
and $C_0$.

\step{Step 1.} We first prove the theorem for $p=1$.
 Using the definition of $\expap$ and Theorem
\ref{dascalare}, we
easily obtain by a change of variable
that for any commutator
$Y$ of length $\ell\in\{1,\dots,s\}$ and for all $\psi\in C^1_\Eucl$,
 \begin{equation}\begin{aligned} \label{daeo}
  \frac{d}{dh}\psi (\eap^{h   Y }(x))& =   Y\psi(\eap^{h   Y}(x))+
   \sum_{\ell_j=\ell+1}^s \a_j(h)   Y_k\psi(\eap^{h   Y}x)
 \\& \qquad
+ \abs{h}^{(s+1-\ell)/\ell} \sum_{i=1}^q
 b_i(x,h)  Y_i\psi(\eap^{h  Y}x),
 \end{aligned}\end{equation}
for all $x\in K$ and $0<\abs{h}\le C^{-1} $,
where the sum is empty if $\ell=s$. If $\ell< s$, then  $\a_j(h)= \ell^{-1} a_j
h^{(\ell_j-\ell)/\ell}$ if $h>0$,  while
$\a_j(h)=- \ell^{-1}\ol a_j
h^{(\ell_j-\ell)/\ell}$ if $h<0$.
The functions $a_j$ come from the statement of Theorem \ref{dascalare}.
 The functions $b_i(x,h)$ can be discontinuous,
if we pass from $h>0$ to $h<0$, but  we have estimate $\abs{b_i(x,h)}\le C$
uniformly in $x,h$.

To complete Step 1, we need to show that  the function
$h\mapsto\frac{d}{dh} \eap^{h Y}z$ is
continuous
for all fixed
$z\in \Omega$. Continuity at any  $h\neq 0$ (say $h>0$) follows immediately from the
decomposition
$\eap^{h  Y} = e^{h^{1/\ell}  Z_1}\cdots e^{h^{1/\ell}  Z_\nu}
$, where $Z_j\in\pm\H$. We show now  continuity at $h=0$. Formula \eqref{daeo}
gives
$ \Bigl| \frac{\p}{\p h}\eap^{h   Y}z -   g(\eap^{h   Y}z) \Bigr|\le
C\abs{h}^{1/\ell}$ (recall notation $Y=:g\cdot\nabla$).
 Therefore, using the l'H\^opital's rule,  we get
\begin{equation*}
\begin{aligned}
\frac{d}{d h}\eap^{h  Y}z\bigr|_{h=0} & := \lim_{h\to 0}\frac{\eap^{h  Y}z -
z }{h}
=\lim_{h\to 0} g(\eap^{h Y}z)+
O(\abs{h}^{1/\ell})
= g(z),
\end{aligned}
\end{equation*}
where we need  the $d$-continuity of $  g$. This shows existence of the
derivative at $h=0$. To see continuity, just let
$h\to 0$ in \eqref{daeo}.

\step{Step 2.}  By induction on $p$, we show that $E$ is $C^1$ smooth.
Assume that  $(h_1, \dots, h_{p-1})\mapsto 
\eap^{h_1 U_1}\cdots \eap^{h_{p-1}U_{p-1}}(x)$ is $C^1$ for all choice of $U_1, \dots, U_{p-1}$. We need to show that $(h_1, \dots, h_{p})\mapsto 
\eap^{h_1 U_1}\cdots \eap^{h_{p}U_{p}}(x)$ is $C^1$ smooth. 

Let $U_1, \dots, U_p\in \P$. First of all 
we show that   the map  $(h_1, \dots, h_p)\mapsto E_*(\p_{h_1})$ is
continuous.
   If $h_1\neq 0$, say $h_1>0$,
 then we decompose for suitable $Z_1, \dots, Z_\mu\in \H$,
\begin{equation*}
 \eap^{h_1 U_1} \cdots \eap^{h_p U_p}x = e^{ h_1^{1/d_1}Z_1}
\cdots  e^{ {h_1}^{1/d_1}Z_\mu}
\eap^{h_2 U_2}\cdots\eap^{h_p U_p}x.
\end{equation*}
Note that by standard ODE theory, the map $(\t_1, \dots,\t_\mu,z )\mapsto
e^{\t_1 Z_1}\cdots e^{\t_\mu Z_\mu}z$ is $C^1$.
Therefore, by means of  Lemma \ref{torrette}, we have existence and continuity
of $\p_1 E(h)= E_*(\p_{h_1})$ at any point
 of the form $h= (h_1, h_2, \dots, h_p)$ with $h_1\neq 0$.

To discuss the case $h_1=0$, recall that  formula \eqref{daeo}
gives
\begin{equation*}
\begin{aligned}
& \Bigl| \frac{\p}{\p h_1}\eap^{h_1 U_1}\cdots\eap^{h_p U_p}x -U_1(\eap^{h_1
U_1}\cdots\eap^{h_p U_p}x)\Bigr|
\le C\abs{h_1}^{1/d_1}.
\end{aligned}
\end{equation*}
 Therefore, using de~l'H\^opital's rule, for all $h=(0,h_2, \dots, h_p)=:(0,\wh
h_1)$, we get
\begin{equation*}
\begin{aligned}
\p_1E(0,\wh h_1)
&:=\lim_{h_1\to 0}\frac{\eap^{h_1 U_1}\eap^{h_2
U_2}\cdots\eap^{h_p U_p}x-\eap^{h_2 U_2}\cdots\eap^{h_p U_p}x}{h_1}
\\&=\lim_{h_1\to 0}U_1(\eap^{h_1 U_1}\eap^{h_2
U_2}\cdots\eap^{h_p U_p}x)+
O(\abs{h_1}^{1/d_1})
= U_1(E(0, \wh h_1)),
\end{aligned}
\end{equation*}
where we need  the $d$-continuity of $U_1$.
This shows  existence of $\p_{1}E(0, \wh h_1)$.

To show continuity  of $\p_{h_1}E$ at  $h^* = (0, \wh h_1^*)\in
B_\Eucl(C^{-1})$,
write by expansion \eqref{daeo}
\begin{equation}
\label{sosso} 
\begin{aligned}
 \bigl| \p_1E(h_1, \wh h_1)  & - \p_1E(0,\wh h_1^*)\bigr|
\\
&=
\Bigl|U_1(E(h_1, \wh h_1))+\sum_{d_1 +1\le \ell_j\le s}
\a_j(h_1)Y_j(E(h_1, \wh h_1))
\\&\quad
+
\abs{h_1}^{(s+1-d_1)/d_1}
\sum_{1\le i\le q}  b_i Y_i(E(h_1, \wh h_1))
- U_1(E(0, \wh h_1^*))\Bigr|
\\&\le C\abs{h_1}^{1/d_1}+ \abs{U_1(E(h_1, \wh h_1))
-U_1(E(0, \wh h_1^*))}\to 0,
\end{aligned}
\end{equation}
as $(h_1, \wh h_1)\to (0, \wh h_1^*)$, here we used  assumption
\eqref{sussmann}   for $ U_1$.

To conclude \emph{Step 2}, we show the   continuity of $\p_{h_k} E$ for all $2\le k\le p$. Write
by the
chain rule
\begin{equation}\label{cksh} 
\begin{aligned}
 \frac{\p}{\p h_k} E(h) & = F_{U_1}(h_1, \eap^{h_2 U_2}\cdots(x))\cdots
F_{U_{k-1}}(h_{k-1}, \eap^{h_k U_k}\cdots(x))\frac{\p}{\p h_k}\eap^{h_k
U_k}\cdots(x).
\end{aligned}
\end{equation}
This ends the proof, because the right-hand side depends
continuosly on
$h_1, \dots, h_p$, by Lemma \ref{torrette} and the first part of \emph{Step 2}.

\step{Step 3.}
We show expansion \eqref{nhb} and estimates \eqref{sogliola} and
\eqref{merluzzo} for any $p$ and for all $k=1,\dots, p$.

Let $U_k = Y_{i_k}$,  $d_k:= \ell_{i_k}$ and
$E_{ \langle j,k\rangle  }(x):= \eap^{h_j  U_j }     \cdots
\eap^{h_k  U_k}(x)$ for all $1\le j\le k\le p$.
 We agree that $E_{\langle j,j-1\rangle } $ denotes the identity function.
Observe that the function
$z\mapsto E_{ \langle j,k\rangle  }(z) $ is a $C^1$ diffeomorphism for any
fixed $h_j, h_{j+1}, \dots, h_k$.
Then, for $k\in\{1,\dots,p\}$, we may use \eqref{daeo} and we get
\begin{equation}\label{baleno}
\begin{aligned}
E_*(\p_{h_k}) &= U_k
E_{\langle 1,k-1\rangle }(E_{\langle k,p\rangle}(x))
+\sum_{\ell_j = d_k  +1}^s\a_j(h_k)
 Y_{j}
E_{\langle 1,k-1\rangle }(E_{\langle k, p \rangle}(x))
\\&\quad+ \abs{h_k}^{(s+1-d_k)/d_k}
\sum_{i=1}^q  b_i  Y_i E_{\langle 1,k-1\rangle}(E_{\langle k,  p \rangle } (x)),
\end{aligned}
\end{equation}
where $b_i$ denote bounded functions and $\abs{\alpha_j(h_k)}\le
C\abs{h_k}^{(\ell_j - d_k)/d_k}$.

To get formula \eqref{nhb}, it suffices to use a rough expansion of  each term
as follows. Write for $\lambda\in\{1,\dots,p\}$ and
$h_\lambda >0$,
$\eap^{h_\lambda   U_\lambda} = e^{- h_\lambda^{1/d_\lambda}
Z_1}\cdots
 e^{-h_\lambda^{1/d_\lambda} Z_\nu}$, for suitable $Z_i\in \pm \H$.
Then for all $j\in \{1,\dots,q\}$  write
\begin{align*}
  Y_j(\psi \eap^{h_\lambda   U_\lambda})(z)
&=
   Y_j(\psi  e^{-{h_\lambda}^{1/d_\lambda}
Z_1}\cdots
 e^{-{h_\lambda}^{1/d_\lambda}  Z_\nu})(z)
\\&=    Y_j\psi (\eap^{h_\lambda   U_\lambda} z)+
\sum_{\abs{\a}=1
}^{s-\ell_j}\ad_{  Z_\nu}^{\a_\nu}\cdots\ad_{  Z_1}^{\a_1}  Y_j \psi
(\eap^{h_\lambda   U_\lambda} z)\frac{h_\lambda^{\abs{\a}/d_\lambda}}{\a!}
\\&\qquad
 + O_{s+1}(\abs{h_\lambda}^{(s+1-\ell_j)/d_\lambda}, \psi,
\eap^{h_\lambda   U_\lambda}z )
\\&=   Y_j\psi (\eap^{h_\lambda   U_\lambda} z)
+
\sum_{\ell_i= \ell_j + 1}^s c_i  \abs{h_\lambda}^{(\ell_i -
\ell_j)/d_\lambda}   Y_i\psi (\eap^{h_\lambda  U_\lambda}x)
\\&\qquad \qquad
+\abs{h_\lambda}^{(s+1-\ell_j)/d_\lambda}\sum_{i=1}^q b_i  Y_i\psi
 (\eap^{h_\lambda  U_\lambda}x),
\end{align*}
where we use the pointwise form of the remainder, see the proof of Theorem
\ref{dascalare}. Here $c_i$ are constants, while $b_i$ are bounded functions.
The proof of \eqref{nhb} follows from \eqref{baleno} via a repeated application
of this expansion.
If $h_\lambda<0$, then the terms $c_i$ and $b_i$ may change, but the argument
gives the same conclusion.
The proof of the theorem is concluded
\end{proof}

\begin{remark}\label{tredodici} $\,$\begin{itemize*} \item[(i)]    Let $X_w$ be a commutator of length $\abs{w}\le s$.
Define the function $
H(t,x):=\frac{d}{dt} \eap^{t X_w}(x).$
 Under our assumptions $\A_s$
we  may claim that $ H(t,x)$ exists  for all $(t,x)$.
 However, we can not expect  that the function
$(t,x)\mapsto H(t,x)$ is continuous in
$(-t_0, t_0)\times\Omega$.
Indeed, in order to show the
continuity of $H$ at a point $(0, \wt x)$, because 
\begin{equation*}
\begin{aligned}
\abs{H(t,x)- H(0,\wt x)}&\le\abs{H(t,x)-
H(0,x)}+\abs{H(0,x)- H(0, \wt x)}.
\\&=\bigl|\frac{d}{dt}\eap^{t X_w}x - f_w(x)\bigr|+\abs{f_w(x)- f_w(\wt x)}.
\end{aligned}
\end{equation*}
The first term  can be made   small   uniformly in $x$, if $\abs{t}$ is small.
In order to make the second term  small, we can use only assumption
\eqref{sussmann}, which does not ensure any continuity if    $x$ and $
\wt x$ belong to different orbits. 

\item[(ii)] Under our assumptions, we cannot expect that 
 maps $h\mapsto E_{I,x}(h)$ 
are   more than $C^1$. Indeed, the term $F_{U_1}(h_1, \eap^{h_2 U_2} \cdots \eap^{h_p U_p}x)$ in \eqref{cksh} depends continuously on 
$h_2, \dots, h_p$, if $\H$ is a $C^1$ family (recall that $F_{U_1}(h,x):=\nabla \eap^{h U_1}(\xi)$ is only continuous in $\xi$). An inspection of the proof above shows that if $\H$ is a $C^2$ family and $\A_s$ holds, then $E_{I,x}\in C^{1, 1/s}_\loc$, but this regularity cannot be improved, even if $X_j\in C^{\infty}$ or $ C^\omega$; see \cite[Example~5.7]{MM}.
\end{itemize*} \end{remark}

Now we can easily prove the regularity of orbits, along
the lines of the proof in
\cite{Agrachev}.
 \begin{theorem}[Regularity of $\A_s$ orbits]\label{frr}
Let $\H$ be a  system   of  $\A_s$   vector fields. Then  each orbit    $\O$
with the   topology $\t_d$ is a connected 
 $C^1$ smooth immersed  submanifold of $\R^n$ satisfying 
  $T_x\O= P_x:=\Span\{X_w(x):1\le \abs{w}\le s\}$
 for all $x\in \O$.  
\end{theorem}
\begin{proof}  Let $x_0\in \R^n$ and let  $\O:= \O_\H^{x_0}$ be its $\H$-orbit.
We know from Remark \ref{osservo} that $\dim P_x = \dim P_{x_0} =:p$ is constant
in $\O$. For each $x\in \O$ choose $I\in \I(p, q)$ such that $\abs{Y_I(x)} \neq
0 $. By Theorem \ref{deduciamo} and by the implicit function theorem, we may
claim
that for a suitable $ O_{I,x}\subset\R^p$,  open neighborhood of the
origin, the map $E_{I,x}:O_{I,x}\to\R^n$   is  a $C^1$
full-rank map which parametrizes a $C^1$ smooth, $p$-dimensional embedded
submanifold   $E_{I,x}(O_{I,x})\subset \R^n$.
Note also that 
 $E_{I,x}(O_{I,x})\subset\O$ and,
by Theorem~\ref{deduciamo}, $T_{E_{I,x}(h)} E_{I,x}(O_{I,x}) = P_{E_{I,x}(h)}$,
for all $h\in O_{I,x}$.
Let
\[ \begin{aligned} 
\U &:=\{ E_{I, x}(O):  x\in \O, I\in
\I(p,q),\abs{Y_I(x)}\neq 0   
\\ &\qquad \qquad \text{and  $O\subset O_{I, x}$ is a  open
neighborhood
of the origin$\}.$}
\end{aligned}\]
We claim that the family $\U$ can be used as a base for a topology   $\t(\cal{U})$   on $\O$.
To see that, we need to show that if the intersection of the $p$-dimensional submanifolds
$E_{I,x}(O)$ and $ E_{I',x'}(O')$ is
nonempty, then it contains  a small manifold  of the form
$E_{I'',x''}(O'')$, if $O''$
is a sufficiently small neighborhood of the origin.
Let $\Sigma:=E_{I,x}(O)$ and 
$\Sigma'= E_{I',x'}(O')$ and let $x''\in\Sigma\cap\Sigma'$. Recall that both $\Sigma$ and $\Sigma'$
are embedded $C^1$ submanifolds of $\R^n$. Let $I''\in \I(p,q)$ be such that $\abs{Y_{I''}(x'')}\neq0$.
Let $O''\subset\R^p$ be a small open neighborhood of the origin. For any $h\in O''$, the point $E_{I'',x''}
(h) $ can be written as $ e^{\t_1Z_1}\cdots e^{\t_\nu Z_\nu}x$ where $Z_j\in \pm\H$ and $\sum_j\abs{\t_j}\le C\|h\|_I$.
By a repeated application of   Bony's theorem \cite[Theorem 2.1]{Bony69}, it follows that $E(h)\in\Sigma$, provided that $h$ is sufficiently close to the origin. The same argument applies to $\Sigma'$.
Thus we have proved that $\U$ can be used as a topology base.

A similar argument   shows that any submanifold of the form $E_{I, x}(O)\in\U$ 
contains a small ball $B_d(x, \sigma)$. Therefore $\t_d$ is stronger than $\t(\U)$. The fact that $\t(\U)$ is 
stronger that $\t_d$  follows easily from estimate $d(E_{I,x}(h),x)\le C\|h\|_I.$ 
Finally, since all paths of the form $t\mapsto e^{tZ}x\in(\O,\t(\U))=(\O,\t_d)$ are continuous,   the orbit is connected.

The $C^1$ differential
structure on $\O$ is given by the family maps $E_{I,x}\bigr|_{O}$ where $x\in
\O$, $I\in\I(p_x, q)$ is such that $\abs{Y_I(x)}\neq 0$ and $O\subset O_{I,x}$ is an  open neighborhood of the origin.
\end{proof}

\begin{example}
      \label{esempietto} Let us consider in $\R^3$
the    family $\H=\{X_1, X_2, X_3\}$: 
\[X_1= a(t)\p_x\quad X_2 = x a(t)\p_y \quad\text{and} \quad X_3 = t\p_t,\]
where the function  $a$ satisfies $a(t) = 1+t^3 \sin\big(\frac 1t\big)$, if  $0<\abs{t}<1$,
$a(0)=0$, 
$a\in C^\infty(\R\setminus\{0\})$ and $\inf_\R a>0$.
Note that   $X_j\in C^1_\Eucl(\R^3)$ and 
\[
      [X_1, X_2]= a(t)^2 \p_y,\quad [X_1, X_3] = - t a'(t) \p_{x}\quad\text{and}\quad
[X_2, X_3] = -t a'(t) x\p_{y}.
\]
If $0<
\abs{t}<1$, then 
\[
 \frac{d}{dt} (ta'(t)) = \frac{d}{dt}\Big(3t^3 \sin \frac 1t  - t^2\cos
\frac 1t\Big)
=9t^2 \sin\frac 1t - 5t\cos\frac 1t - \sin\frac{1}{t}
\]
is discontinuous at $t=0$. Therefore $X_{13}$ and $X_{23}\notin C^1_\Eucl$ and the $C^1$ singular Frobenius theorem does not apply to the family $\P=\{ X_1, X_2, X_3, [X_1, X_2], [X_1, X_3], [X_2, X_3]\}$.

However,  
we claim that the family $\H$ belongs to  our class  $\A_2$. 
To show this claim, we first prove that $X_j\in C^{1,1}_{\H,\loc}$ . To see that,  
 it suffices to show that $X_3^\sharp X_3^\sharp a \in C^0_\H$. But, if $0<\abs{t}<1$, we have
\begin{equation}\label{continuous} 
 X_3^\sharp
 X_3^\sharp a(t)= t\p_t(ta'(t)) = 9t^3 \sin\frac 1t - 5 t^2\cos\frac 1 t - t\sin\frac 1t,
\end{equation} 
which is a continuous function up to $t=0$ (note that, since  $ X_3^\sharp a(0)=0$, we have 
$X_3^\sharp X_3^\sharp a(0) = \lim_{t\to 0}t^{-1}(X_3^\sharp a (e^{t X_3}(0)) - X_3^\sharp a(0) )= 0$).
Since $X_{12},X_{13}$ and $X_{23}\in C^0_\Eucl$, condition 
 \eqref{sussmann} is fulfilled.

Finally, we have to check the $2$-involutivity, i.e.~that 
for all $i,j,k$ 
we can write
$\ad_{X_i}X_{jk}= \sum_{\abs{w}\le 2}b^w X_w$
with $b^w$ locally bounded.
A computation shows that the nonzero terms are the following (we work with $0<\abs{t}<1$)
\begin{align*}
 -\ad_{X_1}X_{23}&=  \ad_{X_2}X_{13}
  =\frac 12 \ad_{X_3}X_{12}= ta(t)a'(t)\p_y = \frac{ta'(t)}{a(t)}X_{12} 
\\   \ad_{X_3}X_{13} & =-t\p_t(ta'(t))\p_x =\frac{-t\p_t(ta'(t))}{a(t)}X_1
 \\ \ad_{X_3} X_{23}&= -xt\p_t(ta'(t))\p_y =  \frac{-t\p_t(ta'(t))}{a(t)}X_2.
\end{align*}
Since $\inf_{\R} a>0$, one can see with the help of \eqref{continuous}
that both the coefficients $ta'(t)/a(t)$ and $-t\p_t(ta'(t))/a(t)$ are locally bounded.
Thus, hypothesis $\A_2$ is fulfilled and our main theorem applies.

Note finally that it is very easy to see that there are three orbits
of the family $\H$. 
Namely, $\O_1:= \{(x,y,t):t>0\}$, $\O_2= \{t=0\}$ and $\O_3 = \{t<0\}$ and they are  integral manifolds of the distribution generated by the family $\P.$ 
\end{example}

\begin{remark}
      \label{finalremark} 
A natural   question  concerns sharpness  
of the $C^1$ regularity
of $\O_\H$. It is reasonable to  guess that $C^1$ regularity is not sharp.
Actually, we do not have any example of vector fields of class $\A_s$ where the integral manifolds $\O_\H$ are less than $C^2$. 
However, under our assumptions, maps $E_{I,x }$ cannot provide more than $C^1$ regularity, see Remark~\ref{tredodici}-(ii).

A related issue concerns  the   regularity of the orbit $\O_\H$ of a generic family of $C^1$ (or even
 Lipschitz-continuous)
 vector fields which do not satisfy any involutivity   assumptions.   
This would require a careful discussion of a nonsmooth version of 
Sussmann's orbit theorem. 

We plan to discuss such questions in a future study.
\end{remark}

\appendix
\section{Appendix}
\label{bio}

Here we prove the multilinear algebra lemma which   has been used in the proof of Lemma \ref{lili}.
The same formula is proved by  \cite[Lemma 3.6]{Street}, but  here we exploit a
slightly different argument,   which does not rely   on the
formalism of  Lie derivatives.
\begin{lemma}[Linear algebra]\label{ollo}
 Let $p\le n$  and let $U_1,\dots,U_p$ be constant vector fields in $\R^n$.
 Let $Z= \sum_{\b=1}^n f^\b\p_\b\in C^1_\Eucl$.
Then, for any $( k_1, \dots, k_p)\in \I(p, n)$,
\begin{equation}\label{olletto}
\begin{aligned}
&\sum_{\a=1}^p dx^{k_1}  \wedge\cdots  dx^{k_p} \Big(U_1, \dots,
U_{\a-1}, \sum_{\b=1}^n U_\a f^\b  \p_\b ,U_{\a+1}, \dots, U_p \Big)
\\& =   \sum_{\g=1}^n \sum_{\b=1}^p \p_\g f^{k_\b}
dx^{(k_1,\dots, k_{\b-1})}  \w  dx^\g  \wedge
dx^{(k_{\b+1},\dots, k_p)}(U_1, \dots, U_p).
\end{aligned}
\end{equation}
\end{lemma}
Note that in the particular case $p=n$,  the right-hand side is
$\Div (f)\det[U_1, \dots, U_n]$.
\begin{proof}
 Recall first that if we are given  $(V_\a^\b)_{\a,\b }\in \R^{p\times p}$, then
the matrix 
$(\cof V)_\a^\b := \det[V_1, \dots, V_{\a-1}, \p_\b, V_{\a+1}, \dots]$
satisfies
\begin{equation}\label{ladro}
 \sum_{\mu=1}^p V_{\mu}^\sigma (\cof V)_\mu^\rho  = (\det V )\delta_{\sigma\rho}
\end{equation}

To prove the lemma,
observe first that $dx^{k_\mu}(\p_\b)=0$ if $\mu\in \{1,\dots,p\}$ and
 $\b\notin
\{k_1, \dots, k_p\}$. Therefore the left-hand side of \eqref{olletto} takes the
form
\begin{align*}
  \sum_{\a=1}^p  & dx^{k_1}  \wedge dx^{k_p} \Big(U_1, \dots,
U_{\a-1}, \sum_{\b=1}^p U_\a f^{k_\b} \p_{k_\b},U_{\a+1}, \dots, U_p \Big)
\\&=  \sum_{\substack{\a,\b =1,\dots, p\\
\g=1,\dots, n}}
U_\a^\g \p_\g f^{k_\b}
dx^{k_1}  \wedge dx^{k_p} \big(U_1, \dots, U_{\a-1},     \p_{k_\b},U_{\a+1},
\dots, U_p \big)
\\&=
   \sum_{\substack{\b =1,\dots, p\\
\g=1,\dots, n}}
 \p_\g f^{k_\b}\sum_{\a=1}^p   U_\a^\g
\cof
\left[
\begin{smallmatrix}
  U_1^{k_1}& \dots & U_{p}^{k_1}\\
 \vdots & \vdots& \vdots\\
U_1^{k_p}& \dots & U_{p}^{k_p}
\end{smallmatrix}
\right]_\a^{\b}
\stackrel{\eqref{ladro}}{=}
\sum_{\substack{\b =1,\dots, p\\
\g=1,\dots, n}}    \p_\g f^{k_\b} \det
 \left[\begin{smallmatrix}
 U_1^{k_1}& \dots & U_{p}^{k_1}\\
 \vdots & \vdots& \vdots
\\  U_1^{k_{\b-1}}&\cdots  & U_{p}^{k_{\b-1}}
\\ U_1^\g &\cdots & U_p^\g
\\  U_1^{k_{\b+1}}&\cdots  & U_{p}^{k_{\b+1}}
\\
 \vdots & \vdots& \vdots
\\   U_1^{k_{p}}&\cdots  & U_{p}^{k_{p}}
\end{smallmatrix}\right]
 \\&= \sum_{\substack{\b =1,\dots, p\\
\g=1,\dots, n}}
\p_\g f^{k_\b} dx^{(k_1,\dots,k_{\b-1})} \w  dx^\g  \wedge
dx^{(k_{\b+1},\dots,k_p)}(U_1, \dots, U_p),
\end{align*}
as desired.
\end{proof}

\footnotesize

 \phantomsection
\addcontentsline{toc}{section}{References}  
\def\cprime{$'$}
\providecommand{\bysame}{\leavevmode\hbox to3em{\hrulefill}\thinspace}
\providecommand{\MR}{\relax\ifhmode\unskip\space\fi MR }
\providecommand{\MRhref}[2]{%
  \href{http://www.ams.org/mathscinet-getitem?mr=#1}{#2}
}
\providecommand{\href}[2]{#2}



\normalsize
\bigskip \noindent\sc \small  Annamaria Montanari, Daniele Morbidelli
\\ Dipartimento di Matematica,
Universit\`{a} di Bologna  (Italy)
\\Email: \tt   annamaria.montanari@unibo.it,
daniele.morbidelli@unibo.it

\end{document}